\documentclass{amsart}

\usepackage{amssymb}
\usepackage{eucal}
\usepackage{color}
\usepackage{arydshln}
\usepackage{verbatim}

\usepackage{multirow,bigdelim}
\usepackage{amsfonts}
\usepackage{dsfont}
\usepackage{amsmath}
\usepackage{mathrsfs}
\usepackage{todonotes}
\usepackage{hyperref}
\usepackage{multimedia}
\usepackage{graphicx}
\usepackage{indentfirst}
\usepackage{bm}
\usepackage{amsthm}
\usepackage{mathrsfs}
\usepackage{latexsym}
\usepackage{amsmath}
\usepackage{amssymb}
\usepackage{microtype}
\usepackage{relsize}
\usepackage{hyperref}
\usepackage{tikz-cd}
\usepackage[backend=biber,style=numeric,autocite=superscript]{biblatex}
\allowdisplaybreaks[4]

\numberwithin{equation}{subsection}

\DeclareSymbolFont{SY}{U}{psy}{m}{n}
\DeclareMathSymbol{\emptyset}{\mathord}{SY}{'306}

\theoremstyle{plain}

\newtheorem{theorem}{Theorem}[section]
\newtheorem{corollary}[theorem]{Corollary}
\newtheorem{lemma}[theorem]{Lemma}
\newtheorem{proposition}[theorem]{Proposition}
\theoremstyle{definition}
\newtheorem{definition}[theorem]{Definition}
\newtheorem{remark}[theorem]{Remark}
\newtheorem{example}[theorem]{Example}

\numberwithin{equation}{section}

\makeatletter
\@namedef{subjclassname@2020}{\textup{2020} Mathematics Subject Classification}
\makeatother
\begin{document}

	\title[Geometric operators and the third problem of Halmos]{Geometric operators and the third problem of Halmos}

\author{ KUI JI, CHUNLAN JIANG, and ZHENYUE WANG }

\thanks{}

\subjclass[2020]{}

\curraddr[ ]{School of Mathematical Sciences, Hebei Normal University,
	Shijiazhuang, Hebei 050024, China}

\email[]{}

\email[K.Ji]{jikui@hebtu.edu.cn} \email[C.L.Jiang]{cljiang@hebtu.edu.cn} \email[Z.Y.Wang]{wangzhenyuemath@126.com}

\begin{abstract}
	In 1970, P. R. Halmos raised ten open problems on operator theory. The third problem of Halmos is the following:  if an intransitive operator has an inverse, is its inverse also intransitive?
	 This question is reduced to the case of conjugate bi‑geometric operators in \cite{fang2026invariantII}. In the framework of conjugate bi-geometric operators, using tools from complex geometry, operator theory, and function theory, this paper gives an affirmative answer to the third problem of Halmos in the following cases:
	\begin{enumerate}
		\item[(a)] bundle-preserving operators;
		\item[(b)] operators satisfying suitable subexponential
		power growth conditions.
	\end{enumerate}
	These two cases cover a large class of operators.
	
\end{abstract}
\maketitle


\section{Introduction}

Let $\mathcal{H}$ be a complex separable infinite‑dimensional Hilbert space. Denote by $\mathcal{B}(\mathcal{H})$ the set of all bounded linear operators on $\mathcal{H}$. An operator $T\in\mathcal{B}(\mathcal{H})$ is said to be {\it intransitive} if it admits an invariant subspace different from $\{0\}$ and the whole space $\mathcal{H}$; otherwise, $T$ is called a {\it transitive operator}.

\medskip

{\bf Invariant subspace problem:} Given $T\in\mathcal{B}(\mathcal{H})$, does $T$ possess a nontrivial invariant subspace?

\medskip

The invariant subspace problem has been proposed for more than a century, and it remains an open problem to this day.

In the 1930s, von Neumann proposed an elegant method for constructing invariant subspaces of compact operators by approximating with invariant subspaces of a sequence of finite‑rank operators.
In 1954, Aronszajn and Smith generalized the above result to separable Banach spaces \cite{3aronszajn1954invariant}. Wermer \cite{wermer1952existence} proved that every invertible operator \(T\) satisfying
\[
\sum_{n=-\infty}^{\infty}\frac{\log\|T^n\|}{1+n^2}<\infty
\]
and whose spectrum consists of more than one point possesses a nontrivial hyperinvariant subspace.

In 1975, Enflo constructed an operator on a Banach space that possesses no nontrivial invariant subspaces \cite{14enflo1987invariant}\cite{7beauzamy1985operateur}.
During the 1980s, Read further produced a large family of operators without nontrivial invariant subspaces on several non‑reflexive Banach spaces \cite{33read1985l1,34read1986shortproof,35read1988hypercyclic,32read1984solution}.

Closely related to the famous invariant subspace problem is the third problem of Halmos. In 1970, P.~R.~Halmos raised ten open problems in operator theory \cite{halmos1970ten}, among which the third problem of Halmos is stated as follows.

\medskip

{\bf Problem 3:} If an intransitive operator has an inverse, is its inverse also intransitive?

 \vspace{1em}

R.~G.~Douglas, C.~Foias, and C.~Pearcy \cite{8douglas2005common} also pointed out that the third problem of Halmos is the last remaining open problem among Halmos' ten problems.

Reference \cite{halmos1970ten} provides explicit citations for the solutions to all but two of Halmos's ten problems, namely Problems 3 and 6. Problem 6 has been solved by G. Pisier; see \cite{pisier1997polynomially}.

The Gelfand–Hille theorem \cite{gelfand1941char} states that for \(T\in B(X)\) with spectrum \(\sigma(T)=\{1\}\) and a positive integer \(r\), the equality \((T-I)^r=0\) holds if and only if \(\|T^n\|=o(n^r)\), or equivalently \(\|T^n\|=O(n^{r-1})\), as \(|n|\to\infty\). This is of considerable help in addressing Halmos's third problem.

The difficulty in solving the third problem of Halmos is that we only have the 
information that T is invertible. Therefore, to answer this problem aﬃrmatively, 
we need to give the operator $T$ more structure. 

The authors of \cite{fang2026invariantII} made important progress on the third problem of Halmos and introduced the notion of conjugate bi-geometric operators, whose definition is given below:

Let $T\in\mathcal{B}(\mathcal{H})$. We say that $T$ has a conjugate bi‑geometric form, if there is a nontrivial subspace $M$ such that
\[
T=
\begin{pmatrix}
	T_1 & T_{12}\\
	0 & T_2
\end{pmatrix}\!\!\!\!
\begin{matrix}
	M\\
	\ \ M^{\perp}
\end{matrix}
\]
where $T_1^*, T_2\in B_1(O_\delta)$ and $O_\delta\triangleq\{z\in\mathbb{C};\,|z|<\delta\}$ for some $\delta>0$.

And they obtain the following theorem:

\begin{theorem}[Reduction of Halmos' third problem]\cite{fang2026invariantII} 
	\label{thm:1.1}
	The following two statements are equivalent.
	\begin{enumerate}
		\item For any invertible operator $T\in\mathcal{B}(\mathcal{H})$ with a nontrivial invariant subspace, $T^{-1}$ has a nontrivial invariant subspace.
		
		\item For any invertible operator $T\in\mathcal{B}(\mathcal{H})$ with a conjugate bi‑geometric form, $T^{-1}$ has a nontrivial invariant subspace.
	\end{enumerate}
\end{theorem}

Therefore, in this paper, we mainly consider when the conjugate bi-geometric operators (\cite{fang2026invariantII} Definition 5.1) provide an affirmative answer to Halmos's third problem.

In the paper \cite{fang2026invariantII}, the authors showed that the third problem of Halmos can be reduced to the question of whether the inverse of a conjugate bi-geometric operator $T$ has a nontrivial invariant subspace, and implied through concrete examples that if a conjugate bi-geometric operator possesses the property of preserving bundles, then the answer to Halmos' third problem is affirmative. Motivated by this significant progress, we explicitly prove the following theorem.

\noindent\textbf{ THEOREM A} (Theorem \ref{thm:3.7}): Let $T$ be an invertible conjugate bi-geometric operator in $\mathcal{B}(\mathcal{H})$. If $T$ has the bundle-preserving property (see Definition \ref{defin1}), then $T^{-1}$ is intransitive.

\noindent\textbf { THEOREM B} (Theorem \ref{dingli00}): Let
\[
T=\begin{pmatrix}
	A & T_{12}\\
	0 & B
\end{pmatrix}\in \mathcal{B}(\mathcal{H}_1\oplus\mathcal{H}_2),
\]
be an invertible conjugate bi-geometric operator with \(\sigma_{e}(T)=\mathbb{T}\). 
If there exist constants \(M,b>0\) and \(\delta\in (0,1)\) such that
\[
\max\big\{\|A^n\|,\ \|B^n\|,\ \|(A^{*}A)^{-n}\|,\ \|(BB^{*})^{-n}\|\big\}\le M e^{b n^\delta},\quad n\in\mathbb{N}_0,
\]
then \(T\), \(T^{-1}\) have nontrivial hyperinvariant subspaces and $A^{*},B\in B_1(\mathbb{D})$.

\noindent\textbf{THEOREM C} (Theorem \ref{fujihe}): Let
\[
T=\begin{pmatrix}
	A & T_{12}\\
	0 & B
\end{pmatrix}\in B(\mathcal{H}_{1}\oplus \mathcal{H}_{2}),
\]
and suppose that $T^{-1}$ is intransitive. Let
\[
\tilde{T}=\begin{pmatrix}
	\tilde{A} & \tilde{T}_{12}\\
	0 & \tilde{B}
\end{pmatrix}\in B(\tilde{\mathcal{H}}_{1}\oplus \tilde{\mathcal{H}_{2}}),
\]
where $A^{*},\tilde{A}^{*},B,\tilde{B}\in B_{1}(O_{\delta})$.
Denote by $S_{T}$ and $S_{\tilde{T}}$ the induced operators of $T$ and $\tilde{T}$, respectively.
Let $\gamma(w),\tilde{\gamma}(w)$ and $\delta(w),\tilde{\delta}(w)$ be spanning sections for $\ker(S_{T}-w),\ker(S_{\tilde{T}}-w)$ and $\ker(S_{T}^{*}-w),\ker(S_{\tilde{T}}^{*}-w)$, respectively.
If the following conditions hold:
\begin{enumerate}
	\item $\mathcal{K}_{A^{*}}(w)=\mathcal{K}_{\tilde{A}^{*}}(w)$.
	
	\item $\mathcal{K}_{B}(w)=\mathcal{K}_{\tilde{B}}(w)$.
	
	\item
	$
	\dfrac{\bigl|\langle\gamma(w_{1}),T_{12}\delta(\overline{w}_{2})\rangle\bigr|}
	{\|\gamma(w_{1})\|\,\|\delta(\overline{w}_{2})\|}
	=
	\dfrac{\bigl|\langle\tilde{\gamma}(w_{1}),\tilde{T}_{12}\tilde{\delta}(\overline{w}_{2})\rangle\bigr|}
	{\|\tilde{\gamma}(w_{1})\|\,\|\tilde{\delta}(\overline{w}_{2})\|},
	$
	for all $w,w_{1},w_{2}\in O_{\delta}$.
\end{enumerate}
then $\tilde{T}$ is invertible and $\tilde{T}^{-1}$ is intransitive.

This paper is structured as follows. In Section
2, we recall some basic definitions and knowledge about operator theory and complex geometry. In Section 3, we give an affirmative answer to the third problem of Halmos for conjugate bi-geometric operators with the bundle-preserving property. In Section 4, we provide an affirmative answer to Halmos's third problem for conjugate bi-geometric operators whose diagonal blocks, along with their standard left and right inverses, have subexponential power growth. In Section 5, we characterize the scope of the class of operators providing an affirmative answer to Halmos's third problem with the aid of complex geometric methods. In Section 6, we show that \(T^{-1}\) cannot be triangularized by an invertible operator that is upper triangular with respect to the original decomposition. This does not rule out triangularization by a general similarity transformation. Theorems \ref{thm:3.7} and \ref{dingli00} provide sufficient conditions for the intransitivity of \(T^{-1}\).

\section{Preliminaries}

This section introduces several classes of operators related to complex geometry, together with some geometric invariants of operators. The results in the sequel are based on these  basic concepts (\cite{Conway:2000,IJMJjiang2025cowen,article2,cowen1983operators,article2,fang2025invariant,wells2008differential,wendland2005,chahbi2014linear,fang2026invariantII}).

In \cite{article2}, M.~J.~Cowen and R.~G.~Douglas introduced a class of operators, denoted by $B_n(\Omega)$, whose eigenvalues with constant multiplicity $n$ are taken over a domain $\Omega\subset\mathbb{C}$.
The Cowen--Douglas class $B_n(\Omega)$ of operators of rank $n$ is defined as follows \cite{article2}:
\begin{enumerate}
	\item $\Omega\subset \sigma(T)=\{z\in\mathbb{C}\mid T-z \text{ is not invertible}\}$;
	\item $\operatorname{ran}(T-z)=\mathcal{H}$ for all $z\in\Omega$;
	\item $\bigvee_{z\in\Omega}\ker(T-z)=\mathcal{H}$;
	\item $\dim\ker(T-z)=n$ for all $z\in\Omega$.
\end{enumerate}

\begin{definition}\label{def:2.1}\cite{douglas1998}
	Let $T$ be an operator on a Hilbert space $\mathcal{H}$.
	If the range of $T$ is finite‑dimensional, then $T$ is called a 
	\textit{finite‑rank operator}.
	If $T$ maps the unit ball of $\mathcal{H}$ into a compact set, then $T$ is called a \textit{compact operator}.
	Denote by $\mathfrak{F}(\mathcal{H})$ and $\mathcal{K}(\mathcal{H})$ the sets of all finite‑rank operators and compact operators on $\mathcal{H}$, respectively.
\end{definition}

\begin{definition}\label{def:2.2}\cite{douglas1998}
	Let $\mathcal{H}$ be a Hilbert space. The quotient algebra $\mathcal{B}(\mathcal{H})/\mathcal{K}(\mathcal{H})$ is called the \textit{Calkin algebra}.
	Let $\pi$ denote the canonical homomorphism from $\mathcal{B}(\mathcal{H})$ onto $\mathcal{B}(\mathcal{H})/\mathcal{K}(\mathcal{H})$.
	For $T\in\mathcal{B}(\mathcal{H})$, the spectrum of $\pi(T)$ in the Calkin algebra is called the \textit{essential spectrum} of $T$, denoted by $\sigma_e(T)$.
\end{definition}

\begin{definition}\label{def:2.3}\cite{douglas1998}
	Let $\mathcal{H}$ be a Hilbert space. An operator $T\in\mathcal{B}(\mathcal{H})$ is said to be a \textit{Fredholm operator} if $\pi(T)$ is invertible in the Calkin algebra $\mathcal{B}(\mathcal{H})/\mathcal{K}(\mathcal{H})$.
	Let $\mathcal{F}(\mathcal{H})$ stand for the set of all Fredholm operators on $\mathcal{H}$.
\end{definition}
\begin{definition}\cite{chahbi2014linear}
We say that an operator \(T\) is an \textit{essentially unitary operator}, if \(\pi(T)\) is a unitary element in \(\mathcal{B}(\mathcal{H})/\mathcal{K}(\mathcal{H})\).
\end{definition}
The following theorem provides the standard characterization of Fredholm operators.
\begin{theorem}[Atkinson]\label{thm:2.4}\cite{douglas1998}
	Let $\mathcal{H}$ be a Hilbert space. An operator $T\in\mathcal{B}(\mathcal{H})$ is Fredholm if and only if $\operatorname{ran}T$ is closed, and both $\ker T$ and $\ker T^*$ are finite‑dimensional.
\end{theorem}
\begin{definition}\cite{wendland2005}
	A function \(\Phi : \mathbb{R}^s \to \mathbb{R}\) is called \textit{radial} provided there exists a univariate function \(\varphi : [0, \infty) \to \mathbb{R}\) such that
	\[
	\Phi(x) = \varphi(r), \quad \text{where} \quad r = \|x\|,
	\]
	and \(\|\cdot\|\) is some norm on \(\mathbb{R}^s\) -- usually the Euclidean norm.
\end{definition}

\begin{definition}\label{def:1.1}\cite{fang2026invariantII}
	Let $T\in\mathcal{B}(\mathcal{H})$. We say that $T$ has a \textit{conjugate bi‑geometric form}, if there is a nontrivial subspace $M$ such that
	\[
	T=
	\begin{pmatrix}
		T_1 & T_{12}\\
		0 & T_2
	\end{pmatrix}\!\!\!\!
	\begin{matrix}
		M\\
		\ \ M^{\perp}
	\end{matrix}
	\]
	where $T_1^*, T_2\in B_1(O_\delta)$ and $O_\delta\triangleq\{z\in\mathbb{C};\,|z|<\delta\}$ for some $\delta>0$.
\end{definition}

We will show that it suffices to consider invertible operators with conjugate bi‑geometric form for the third problem of Halmos.

\begin{definition}\cite{IJMJjiang2025cowen}
Let \(A\in B_{1}(\Omega)\) with \(0\in \Omega\). Then there exists an orthonormal basis \(\{e_{n}\}_{n=0}^{\infty}\) of \(\mathcal{H}\) such that 
\(e_{n-1}\in \ker A^{n}\ominus\ker A^{n-1}\) for \(n\ge 1\). If \(AB=I\) and 
\[
\langle Be_{j},e_{i}\rangle=\begin{cases}
	0, & i=0,\\[2mm]
	b_{i,\,j}, & i\ge 1,
\end{cases}
\qquad j\ge 0,
\]
then \(B\) is called the \textit{standard right inverse} of \(A\).
		Dually, let $A^*\in B_1(\Omega)$ with $0\in\Omega$.
		If $C^*$ is the standard right inverse of $A^*$,
		then $C$ is called the \textit{standard left inverse} of $A$.
\end{definition}

\begin{definition}\cite{cowen1983operators}
	Let $\mathscr{B}_{m,n}(\Omega)$ denote the class of bounded linear operators on $\mathscr{H}$ defined by

$	\mathscr{B}_{m,n}(\Omega)
	=\bigl\{T\in\mathscr{L}(\mathcal{H}):$
	\begin{align*}
		&(1)\ \dim\ker(T-\omega)=m \text{ and } \dim\ker(T-\omega)^{*}=n \text{ for all } \omega\in\Omega;\\
		&(2)\ (T-\omega)\mathcal{H} \text{ and } (T-\omega)^{*}\mathcal{H} \text{ are closed for all } \omega\in\Omega;\\
		 &(3)\ \mathcal{M}_T+\mathcal{N}_T=\mathcal{H},\ \text{where }
		\mathcal{M}_T=\bigvee_{\omega\in\Omega}\ker(T-\omega),\\
		&\mathcal{N}_T=\bigvee_{\omega\in\Omega}(\ker(T-\omega)^{*})\bigr\}.
	\end{align*}
	
\end{definition}

\begin{remark}\label{rem:2.7}
	Let
	\[
	T=\begin{pmatrix}
		T_1 & X\\
		0 & T_2
	\end{pmatrix},\quad
	\tilde{T}=\begin{pmatrix}
		\tilde{T}_1 & \tilde{X}\\
		0 & \tilde{T}_2
	\end{pmatrix}\in B_{1,1}(\Omega).
	\]
	Let $P_i(w)$ be the orthogonal projection onto $\ker(T - w)^i$ for $i=1,2$, and $Q_j(w)$ the orthogonal projection onto $\ker((T - w)^*)^{j}$ for $j=1,2$. Let $\gamma(w)$ and $\delta(w)$ be sections spanning $\ker(T - w)$ and $\ker(T^* - w)$, respectively. Then we have the following conclusion:

	If $T_1\sim_u \tilde{T}_1$ and $T_2\sim_u \tilde{T}_2$, then $T\sim_u \tilde{T}$ if and only if
	\[
	\frac{|\langle \gamma(\omega_1), X\delta(\overline{\omega}_2)\rangle|}
	{\|\gamma(\omega_1)\|\,\|\delta(\overline{\omega}_2)\|}
	=
	\frac{|\langle \tilde{\gamma}(\omega_1), \tilde{X}\tilde{\delta}(\overline{\omega}_2)\rangle|}
	{\|\tilde{\gamma}(\omega_1)\|\,\|\tilde{\delta}(\overline{\omega}_2)\|}
	\]
	for all $\omega_1,\omega_2\in \Omega$ \cite{cowen1983operators}.
\end{remark}

\begin{definition}\label{def:2.8}
	Let
	\[
	T=\begin{pmatrix}
		A & T_{12}\\
		0 & B
	\end{pmatrix}\in B(M\oplus M^{\perp}),
	\]
	where $A^{*},B\in B_1(O_\delta)$.
	Define its \textit{induced operator}
	\[
	S=\begin{pmatrix}
		S_1 & X\\
		0 & S_2
	\end{pmatrix}
	=\begin{pmatrix}
		A^{*} & T_{12}\\
		0 & B^{*}
	\end{pmatrix}.
	\]
\end{definition}

It is easy to verify that the induced operator $S$ of the conjugate bi-geometric operator $T$ satisfies $S\in B_{1,1}(O_\delta)$.

\begin{definition}\label{def:2.9}
	Let $\mathcal{H}$ be a Hilbert space and let $\{\xi_n\}_{n=1}^{\infty}$ and $\{\eta_n\}_{n=1}^{\infty}$ be two sequences of vectors in $\mathcal{H}$. We call $\{\xi_n\}_{n=1}^{\infty}$ and $\{\eta_n\}_{n=1}^{\infty}$ \textit{biorthogonal sequences} if
	\[
	\langle\xi_m,\eta_n\rangle=\delta_{m,n},\quad \forall\,m,n\in\mathbb{N}.
	\]
	If, furthermore, $\operatorname{\overline{Span}}\{\xi_n\}_{n=1}^{\infty}=\operatorname{\overline{Span}}\{\eta_n\}_{n=1}^{\infty}=\mathcal{H}$, we call $\{\xi_n\}_{n=1}^{\infty}$ and $\{\eta_n\}_{n=1}^{\infty}$ \textit{dual systems}.
\end{definition}
\begin{definition}
Let $T\in \mathcal{B}(\mathcal{H})$. If there exist $C_{1},C_{2}\ge 0$ and $\alpha\in (0,1)$ such that $\|T^{n}\|\le C_{1}e^{C_{2}n^{\alpha}}, n \in \mathbb{N}_{0}$, then $T$ is said to have \textit{subexponential power growth}.

\end{definition}

 In what follows, we provide an introduction to the complex-geometric content associated with the Cowen--Douglas class of operators.

 Let $\Omega \subset \mathbb C$ be a connected open set. Denote by $z = x + iy$ and $\bar z = x - iy$ the standard complex coordinates on $\Omega$, so that the cotangent bundle $T^*(\Omega)$ admits a local frame $\{dz, d\bar z\}$ on each fibre. Consequently, the exterior algebra bundle $\Lambda^p(T^*(\Omega))$ has a fibre basis consisting of $1$ (for $p=0$), $dz$ and $d\bar z$ (for $p=1$), and $dz \wedge d\bar z$ (for $p=2$) \cite{article2}.

Let $\mathcal E(\Omega)$ denote the algebra of complex-valued $C^\infty$ functions on $\Omega$, and let $\mathcal O(\Omega)$ be its subalgebra of holomorphic functions. For a $C^\infty$ vector bundle $F$ over $\Omega$, we write $\mathcal E(\Omega, F)$ for the space of $C^\infty$ sections of $F$; if $F$ is holomorphic, then $\mathcal O(\Omega, F)$ denotes the space of holomorphic sections.

Let $\mathcal E^p(\Omega) := \mathcal E\bigl(\Omega, \Lambda^p(T^*(\Omega))\bigr)$ denote the space of $C^\infty$ differential $p$-forms on $\Omega$. In particular, we have
\[
\mathcal E^0(\Omega) = \mathcal E(\Omega),\qquad
\mathcal E^1(\Omega) = \{\, f\,dz + g\,d\bar z : f,g \in \mathcal E(\Omega)\,\},\qquad
\mathcal E^2(\Omega) = \{\, f\,dz\wedge d\bar z : f \in \mathcal E(\Omega)\,\}.
\]
We split $\mathcal E^1(\Omega)$ as the direct sum $\mathcal E^{1,0}(\Omega) \oplus \mathcal E^{0,1}(\Omega)$, where
\[
\mathcal E^{1,0}(\Omega) = \{\, f\,dz : f \in \mathcal E(\Omega)\,\},\qquad
\mathcal E^{0,1}(\Omega) = \{\, g\,d\bar z : g \in \mathcal E(\Omega)\,\}.
\]
With the convention that $\mathcal E^{0,0}(\Omega) := \mathcal E(\Omega)$, $\mathcal E^{1,1}(\Omega) := \mathcal E^2(\Omega)$, and $\mathcal E^{2,0}(\Omega) = \mathcal E^{0,2}(\Omega) := \{0\}$, the exterior derivative $d : \mathcal E^p(\Omega) \to \mathcal E^{p+1}(\Omega)$ (for $p=0,1,2$) splits as
\[
d = \partial + \bar\partial,
\]
where $\partial : \mathcal E^{p,q}(\Omega) \to \mathcal E^{p+1,q}(\Omega)$ and $\bar\partial : \mathcal E^{p,q}(\Omega) \to \mathcal E^{p,q+1}(\Omega)$. Moreover, a function $f \in \mathcal E(\Omega)$ is holomorphic if and only if $\bar\partial f = 0$.

\begin{definition}\label{def:2.10}\cite{wells2008differential}
	Let $E\to \Omega$ be a vector bundle. A \textit{Hermitian metric} $h$ on $E$ assigns to each fibre $E_x$ of $E$ a Hermitian inner product $\langle{\cdot,\cdot}\rangle_x$, such that for any open subset $U\subset\Omega$ and any $\xi,\eta\in \mathcal{E}(U,E)$, the function
	\[
	\langle\xi,\eta\rangle\colon U\to \mathbb{C}
	\]
	defined by
	\[
	\langle\xi,\eta\rangle(x)=\langle\xi(x),\eta(x)\rangle_x
	\]
	is of class $C^\infty$.
\end{definition}

\begin{definition}\label{def:2.11}\cite{wells2008differential}
	Let $E\to\Omega$ be a vector bundle. A \textit{connection} $D$ on $E\to\Omega$ is a $\mathbb{C}$‑linear map
	\[
	D\colon \mathcal{E}(\Omega,E)\to \mathcal{E}^1(\Omega,E),
	\]
	satisfying
\[
		D(\varphi\xi)=d\varphi\cdot\xi+\varphi D\xi,
		\label{eq:connection-leibniz}
\]
	where $\varphi\in\mathcal{E}(\Omega)$ and $\xi\in\mathcal{E}(\Omega,E)$.
\end{definition}




\begin{definition}\cite{wells2008differential}
	Let $D$ be a connection on the vector bundle $E\to\Omega$. The \textit{curvature} $\Theta_E(D)$ is defined to be the element
	\[
	\Theta\in\mathcal{E}^2\bigl(\Omega,\operatorname{Hom}(E,E)\bigr),
	\]
	the $\mathbb{C}$‑linear map
	\[
	\Theta\colon \mathcal{E}(\Omega,E)\to\mathcal{E}^2(\Omega,E),
	\]
	which has the local representation with respect to a frame
	\[
	\Theta(f)=\Theta(D,f)=d\theta(f)+\theta(f)\wedge\theta(f).
	\]
\end{definition}

In \cite{article2}, the curvature of a connection is often defined as the $C^\infty$ section
\[
K(E,D) := D^2
\]
of the bundle $\operatorname{Hom}(E, E \otimes \Lambda^2(T^*\Omega))$. It will be convenient to introduce an associated curvature operator, denoted by $\mathcal K(E,D)$, which is a $C^\infty$ section of $\operatorname{Hom}(E, E)$ and is characterized by
\[K(E,D)\xi = \mathcal K(E,D)\xi \otimes (dz \wedge d\bar z)
\quad \text{for all } \xi \in \mathcal E(\Omega, E).
\]

Using the metric compatibility of the Chern connection $D$, for arbitrary sections $\xi, \tilde \xi \in \mathcal E(\Omega, E)$, we have $\langle D^2\xi,\tilde{\xi}\rangle+\langle\xi,D^2\tilde{\xi}\rangle=0.$  Hence we obtain the relation
\[
\langle \mathcal K \xi, \tilde\xi \rangle = \langle \xi, \mathcal K \tilde\xi \rangle,
\]
or equivalently, in terms of the $(1,1)$-form $dz \wedge d\bar z$,
\[
\langle \mathcal K \xi, \tilde\xi \rangle \, dz \wedge d\bar z
= \langle \xi, \mathcal K \tilde\xi \rangle \, dz \wedge d\bar z.
\]
Consequently, $\mathcal K$ is a self-adjoint operator. Moreover, we compute the curvature as follows. For $\xi \in \mathcal E(\Omega, E)$,
\[
K(E,D)\xi= \bar\partial\bigl(h^{-1}\partial h\bigr)\xi.
\] This may be written as
\[
\mathcal K(E,D)\xi(z)
= - \frac{\partial}{\partial \bar z}
\left( h(z)^{-1} \frac{\partial h(z)}{\partial z} \right) \xi(z),
\quad z \in \Omega.
\]

\begin{definition}\cite{wells2008differential}
	For an $\mathcal{S}$‑bundle $E\xrightarrow{\pi}\Omega$, an $\mathcal{S}$‑\textit{section} is an $\mathcal{S}$‑morphism $s\colon \Omega\to E$ satisfying
	\[
	\pi\circ s=\mathbf{1}_\Omega,
	\]
	where $\mathbf{1}_\Omega$ denotes the identity mapping on $\Omega$. That is, $s$ maps each point of the base space into the fibre over that point.
\end{definition}

\section{The third problem of Halmos for bundle-preserving operators}
Halmos's third problem remains open for a wide range of operator classes. In this section, we establish an affirmative answer for a broad class of operators, including cases for which the problem has not previously been resolved. Our main result is stated in Theorem \ref{thm:3.7}.

In view of Theorem \ref{thm:1.1}, we may restrict attention to an invertible operator of the form
\[
T=\begin{pmatrix}
	A & T_{12}\\
	0 & B
\end{pmatrix}\in \mathcal{B}(M\oplus M^\perp),
\]
where $A^{*},B\in B_1(O_\delta)$. It remains to show that $T^{-1}$ is intransitive.

A natural strategy for proving the existence of invariant subspaces of an operator is to determine the explicit form of the operator. For this purpose, it is necessary to characterize the inverse \(T^{-1}\) of a conjugate bi-geometric operator \(T\). The authors of \cite{fang2026invariantII} made substantial progress in this direction and established the following lemma, which provides a viable tool for proving the intransitivity of \(T^{-1}\).
	\begin{lemma}\label{keni}\cite{fang2026invariantII}
	Let \(T \in B(\mathcal{H} \oplus \mathcal{H})\) be an invertible conjugate bi-geometric operator given by
	\[
	T = \begin{bmatrix}
		A & T_{12} \\
		0 & B
	\end{bmatrix}.
	\]
	Choose nonzero vectors \(e_0 \in \ker A^*\) and \(f_0 \in \ker B\) such that \(\langle T_{12}f_0, e_0 \rangle = 1\). Let \(A_L^{-1}\) be the unique left inverse of \(A\) such that \(A A_L^{-1} = I - T_{12}f_0 \otimes e_0\). Let \(B_R^{-1}\) be the unique right inverse of \(B\) such that \(B_R^{-1}B = I - f_0 \otimes T_{12}^*e_0\). Then
	\[
	T^{-1} = \begin{bmatrix}
		A_L^{-1} & -A_L^{-1}T_{12}B_R^{-1} \\
		f_0 \otimes e_0 & B_R^{-1}
	\end{bmatrix}.
	\]
	
\end{lemma}
We call \(A_L^{-1}\) and \(B_R^{-1}\) in the above lemma the left and right inverses of \(A\) and \(B\) associated with \(T\), respectively, and hereafter refer to them simply as the left and right inverses.

Before stating the main theorem of this section, we first introduce the definition of a bundle-preserving operator. To this end, we begin with some preliminaries.

The operator $B$ induces a holomorphic vector bundle by \cite{article2}, which we denote by $E_B$.  A direct computation shows that
\[
B\biggl(\sum_{n=0}^{\infty}\lambda^{n}(B_R^{-1})^{n}f_0\biggr)
=\lambda\sum_{n=0}^{\infty}\lambda^{n}(B_R^{-1})^{n}f_0.
\]
Accordingly, define a local holomorphic section of \(E_B\) by
\[
e_B(\lambda)=\sum_{n=0}^{\infty}\lambda^{n}(B_R^{-1})^{n}f_0,\qquad |\lambda|<\min\{\delta,\dfrac{1}{\|B_{R}^{-1}\|}\},
\]
where $f_0\in\ker B$.

The operator $A_L^{-1}$ also induces a holomorphic vector bundle $E_{A_L^{-1}}$ such that
\[
E_{A_L^{-1}}(\lambda)=\overline{\operatorname{Span}}\Bigl\{k(\sum_{n=0}^{\infty}\lambda^{n}A^{n}e_0'):k\in \mathbb{C}\Bigr\},
\]where $e_{0}^{\prime}\in \ker A_{L}^{-1}$.

Similarly, 
\[
A_L^{-1}\biggl(\sum_{n=0}^{\infty}\lambda^{n}A^{n}e_0'\biggr)
=\lambda\sum_{n=0}^{\infty}\lambda^{n}A^{n}e_0'.
\]

Accordingly, define
\[
e_{A_L^{-1}}(\lambda)=\sum_{n=0}^{\infty}\lambda^{n}A^{n}e_0',\qquad |\lambda|<\dfrac{1}{\|A\|}.
\]
Then $e_{A_L^{-1}}(\lambda)$ is a section of the bundle $E_{A_L^{-1}}$.

For any $\alpha \in (0,+\infty)$, the same argument yields a holomorphic vector bundle $E_{A_{L,\alpha}^{-1}}$ satisfying
\[
E_{A_{L,\alpha}^{-1}}(\lambda)=\operatorname{Span}\left\{k\left(\sum_{n=0}^{\infty}\alpha^n\lambda^n A^n e_0'\right):k\in\mathbb C\right\}, \qquad |\lambda|<\dfrac{1}{\alpha \|A\|}.
\]
Its section is given by
\[
e_{A_{L,\alpha}^{-1}}(\lambda)=\sum_{n=0}^{\infty}\alpha^n\lambda^n A^n e_0'.
\]
\begin{definition}(bundle-preserving operator)\label{defin1}
	Let 
	\[
	T=\begin{pmatrix}
		A & T_{12}\\
		0 & B
	\end{pmatrix}\in \mathcal{B}(M\oplus M^\perp),
	\]
	where $A^{*},B\in B_1(O_\delta)$. If there exists some $\alpha\in(0,+\infty)$ such that $T_{12}e_B(\lambda)=e_{A_{L,\alpha}^{-1}}(\lambda)$ for all $\lambda\in O_\delta$, then \(T\) is said to have the \textit{bundle-preserving property}; equivalently, \(T\) is called a \textit{bundle-preserving operator}.
\end{definition}

Before stating our main results, we need the following lemma.

\begin{lemma}
	\label{lem:3.4}
Let $A^{*}\in B_{1}(O_{\delta}).$ Suppose $A_L^{-1}A = I$, $0\ne e_0'\in\ker A_L^{-1}$ and $0\ne e_0\in\ker A^*$. Then $\langle e_0',e_0\rangle\neq 0$.
\end{lemma}

\begin{proof}

	Assume that $\langle e'_0, e_0 \rangle = 0$. Since
$e_0 \in \ker A^*, \  A^* \in B_1(O_\delta),$
	it follows that $e'_0 \perp \ker A^*$, hence $e'_0 \in \overline{\operatorname{ran} A}$. As $A$ is Fredholm, its range is closed, so $e'_0 \in \operatorname{ran} A.$

	Thus, there exists a nonzero $x$ such that $Ax = e'_0$. It implies that
	\[
	A_L^{-1}Ax = A_L^{-1}e'_0 = 0.
	\]
	 	Then $x=0$, contradicting the choice of $x\neq0$. Therefore the assumption is false, and  $\langle e'_0, e_0 \rangle \neq 0$. 
\end{proof}
\begin{lemma}
	\label{lem:3.5}
	For $\{A^n e_0'\}_{n=0}^{\infty}$, there exists a sequence $\{g_n\}_{n=0}^{\infty}\subset \bigvee_{n\ge 0}\{A^{n}e_{0}^{\prime}\}$ such that $\langle A^{n}e_{0}^{\prime}, g_{m}\rangle=\delta_{n,m}$, where $\delta_{n,m}=
	\begin{cases}
		1,& n=m\\
		0,& n\neq m
	\end{cases}$.
\end{lemma}

\begin{proof}
Let
$
x \in \bigcap_{k=0}^{\infty} \operatorname{\overline{Span}}\{A^n e_0' : n \ge k\}.
$
Then for each $k \in \mathbb{N}_{0}$, we have
\[
x \in \operatorname{\overline{Span}}\{A^n e_0' : n \ge k\}.
\]
Since $A^n e_0' \in \operatorname{\overline{ran}} A^k$ whenever $n \ge k$, it follows that
$
x \in \operatorname{\overline{ran}} A^k.
$
As $A$ is Fredholm, $\operatorname{ran} A^k$ is closed, so $x \in \operatorname{ran} A^k$. Thus
\[
x \in \bigcap_{n=0}^{\infty} \operatorname{ran} A^n.
\]
Since $A^* \in B_1(O_\delta)$, it is well known that
$
\bigcap_{n=0}^{\infty} \operatorname{ran} A^n = \{0\}.
$
Consequently,
\[
\bigcap_{k=0}^{\infty} \operatorname{\overline{Span}}\{A^n e_0' : n \ge k\} = \{0\}.
\]

For each $k \in \mathbb{N}_{0}$, set
\[
\mathcal H_k := \overline{\operatorname{Span}}\{A^n e_0' : n \ge k\}.
\]
Assume that $\bigcap_{k=0}^{\infty} \mathcal H_k = \{0\}$. We claim that the sequence $\mathcal H_k$ is strictly decreasing, i.e.
\[
\mathcal H_0 \supsetneq \mathcal H_1 \supsetneq \mathcal H_2 \supsetneq \cdots .
\]
Suppose that there exists some $k$ such that $\mathcal H_k = \mathcal H_{k+1}$. Since $A^k e_0' \in \mathcal H_k = \mathcal H_{k+1}$, we have
\[
A^{k+1} e_0' = A(A^k e_0') \in A(\mathcal H_{k+1}) \subset \mathcal H_{k+2}.
\]
As $\mathcal H_{k+1}$ is generated by $\{A^{k+1} e_0', A^{k+2} e_0', \ldots\}$, and $A^{k+1} e_0' \in \mathcal H_{k+2}$, it follows that
$
\mathcal H_{k+1} = \mathcal H_{k+2}.
$
 Repeating the same argument inductively yields
\[
\mathcal H_k = \mathcal H_{k+1} = \mathcal H_{k+2} = \cdots .
\]
Consequently, $\bigcap_{j=0}^{\infty} \mathcal H_j = \mathcal H_k$, which contradicts the assumption. Hence $\mathcal H_0 \supsetneq \mathcal H_1 \supsetneq \mathcal H_2 \supsetneq \cdots$.
	
	We now prove that for any $N\in\mathbb{N}_{0}$,
	\[
	\mathcal{H}_{0}=\operatorname{\overline{Span}}\{A^k e_0': 0\le k\le N\}\dot{+}\mathcal{H}_{N+1}.
	\]
	First, we show $\operatorname{\overline{Span}}\{A^{n}e_{0}^{\prime}: 0\le n\le N\}\cap \mathcal{H}_{N+1}=\{0\}$.
	Suppose $w\in \operatorname{\overline{Span}}\{A^{n}e_{0}^{\prime}: 0 \le n\le N\}$ and $v\in \mathcal{H}_{N+1}$ satisfy $w=v$. If $w\neq 0$, then $w=\sum_{i=0}^{N}\alpha_iA^{i}e_{0}^{\prime}=v$. We may assume that $\alpha_j=-1$ and $\alpha_i=0$ for $0 \le i<j$. Then
	\[
	A^{j}e_{0}^{\prime}=\sum_{i=j+1}^{N}\alpha_iA^{i}e_{0}^{\prime}-v\in \mathcal{H}_{j+1}.
	\]
	This implies that $\mathcal{H}_j=\mathcal{H}_{j+1}$. This contradicts $\mathcal{H}_0\supsetneq \mathcal{H}_1\supsetneq \mathcal{H}_2\supsetneq\cdots$. 
	
	Second, we show for every $x\in \mathcal{H}_{0}$, there is $y\in \operatorname{\overline{Span}}\{A^{n}e_{0}^{\prime}: 0 \le n\le N\}$ and $z\in \mathcal{H}_{N+1}$ such that $x=y+z$. Note that there exists a sequence of vectors $\{x_n\}$ in $\mathcal{H}$ such that each $x_n$ can be written as a finite linear combination of $\{A^{n}e_{0}^{\prime}\}$ and $\lim_{n\to\infty}\|x_n-x\|=0$. In particular, there is $y_n\in \operatorname{\overline{Span}}\{A^{n}e_{0}^{\prime}: 0 \le n\le N\}$ and $z_n\in \mathcal{H}_{N+1}$ such that $x_n=y_n+z_n$. Let $\pi$ be the quotient map from $\mathcal{H}_{0}$ onto $\mathcal{H}_{0}/\mathcal{H}_{N+1}$. 
	
	Then $\pi(y_n)=\pi(x_n)$ is a Cauchy sequence in $\mathcal{H}_{0}/\mathcal{H}_{N+1}$. Clearly, $\pi$ is a surjective map from $\operatorname{\overline{Span}}\{A^{n}e_{0}^{\prime}: 0 \le n\le N\}$ onto $\operatorname{\overline{Span}}\{\pi(A^{n}e_{0}^{\prime}): 0 \le n\le N\}$. Suppose $\pi(w_0e_{0}^{\prime}+w_1 Ae_{0}^{\prime} +\cdots+w_NA^{N}e_{0}^{\prime})=0$. Then there exists a $v\in \mathcal{H}_{N+1}$ such that $w_0e_{0}^{\prime}+w_1 Ae_{0}^{\prime} +\cdots+w_NA^{N}e_{0}^{\prime}=v$. Thus $w_0e_{0}^{\prime}+w_1 Ae_{0}^{\prime} +\cdots+w_NA^{N}e_{0}^{\prime}=v=0$. So $\pi$ is an injective map from $\operatorname{\overline{Span}}\{A^{n}e_{0}^{\prime}: 0 \le n\le N\}$ onto $\operatorname{\overline{Span}}\{\pi(A^{n}e_{0}^{\prime}): 0 \le n\le N\}$. 
	
	Since $\operatorname{\overline{Span}}\{A^{n}e_{0}^{\prime}: 0\le n\le N\}$ is finite dimensional, $\{y_n\}$ is a Cauchy sequence. Thus $\{z_n\}$ is also a Cauchy sequence. Let $y=\lim_{n\to\infty} y_n$ and $z=\lim_{n\to\infty} z_n$. Then $y\in \operatorname{\overline{Span}}\{A^{n}e_{0}^{\prime}: 0\le n\le N\}$ and $z\in \mathcal{H}_{N+1}$ such that $x=y+z$. This implies $\mathcal{H}_{0}=\operatorname{\overline{Span}}\{A^{n}e_{0}^{\prime}: 0\le n\le N\}\dotplus \mathcal{H}_{N+1}$.
	
	Set $N=0$. Then 
$
	\mathcal{H}_{0} = \operatorname{Span}\{e_0'\} \dot{+}\mathcal H_1
$
	holds. Define the operator
	\[
	P : \mathcal{H}_{0} \to \operatorname{Span}\{e_0'\},
	\]
	by $P(x) = c e_0'$ whenever $x = c e_0' + x_1$ with $x_1 \in \mathcal H_1$. It is clear that $P$ is a bounded idempotent. Now define a linear functional
	\[
	f : \mathcal{H}_{0} \to \mathbb C,
	\]
	by the condition $P(x) = f(x) e_0'$. This definition is well-posed, and $f$ is obviously linear. Since
	\[
	\|f(x) e_0'\| = \|P(x)\| \le M \|x\| \quad (M := \|P\|),
	\]
	we have
$
	|f(x)| \le \frac{M}{\|e_0'\|} \|x\|,
$
hence $f \in \mathcal{H}_{0}^*$. By the Riesz representation theorem, there exists $g_0 \in \mathcal{H}_{0}$ such that
	\[
	f(y) = \langle y, g_0 \rangle \quad (y \in \mathcal{H}_{0}).
	\]
	Consequently,
	\[
	\langle e_0', g_0 \rangle = 1,\qquad \langle A^n e_0', g_0 \rangle = 0 \quad (n \ge 1).
	\]

	For $N=1$, the same argument yields a bounded idempotent
	\[
	Q : \mathcal{H}_{0} \to \operatorname{Span}\{e_0', Ae_0'\},
	\]
	and a linear functional
	\[
	\phi : \operatorname{Span}\{e_0', Ae_0'\} \to \mathbb C,
	\]
	such that
	\[
	\phi(e_0') = 0,\qquad \phi(Ae_0') = 1.
	\]
	Since the projection $P_1 : \operatorname{Span}\{e_0', Ae_0'\} \to \overline{\operatorname{Span}}\{e_0'\}$ is bounded, it follows that for any $c_0, c_1 \in \mathbb C$,
	\[
	|\phi(c_0 e_0' + c_1 Ae_0')|
	= |c_1|
	= \frac{\|c_1 Ae_0'\|}{\|Ae_0'\|}
	\le \frac{\|P_1\|}{\|Ae_0'\|} \cdot \|c_0 e_0' + c_1 Ae_0'\|.
	\]

Hence $\phi \circ Q \in \mathcal{H}_{0}^*$, and it satisfies 
$
(\phi \circ Q)(A^n e_0') = 0 \  (n \neq 1),\ 
(\phi \circ Q)(A e_0') = 1.
$
By the Riesz representation theorem, there exists $g_1 \in \mathcal{H}_{0}$ such that
\[
\langle A^n e_0', g_1 \rangle = 0 \quad (n \neq 1),\qquad
\langle A e_0', g_1 \rangle = 1.
\]
Repeating the above argument, we obtain a sequence $\{g_n\}_{n=0}^{\infty} \subset \mathcal{H}_{0}$ such that 
\[
\langle A^m e_0', g_n \rangle = \delta_{mn} \quad (m,n \ge 0).
\]
\end{proof}

\begin{lemma}\label{lem:3.6}\cite{fang2026invariantII}
	Let $A\in B(H^2)$. Then $A$ is intransitive if and only if there exist nonzero vectors $f,g\in H^2$ such that
	\[
	\nu=\sum_{n=0}^{\infty}\langle A^n f,g\rangle z^n
	\]
	is a rational function.
\end{lemma}

Motivated by the work of \cite{fang2026invariantII}, we combine techniques from bundle theory with the strategy developed therein to establish Theorem \ref{thm:3.7}.
\begin{theorem}\label{thm:3.7}
Let $T$ be an invertible conjugate bi-geometric operator in $\mathcal{B}(\mathcal{H})$. If $T$ has the bundle-preserving property, then $T^{-1}$ is intransitive.
\end{theorem}
\begin{proof}
	Since $T$ is a conjugate bi-geometric operator,
it can be represented as
	\[
	T=
	\begin{pmatrix}
		A & T_{12}\\
		0 & B
	\end{pmatrix}\in \mathcal{B}(M\oplus M^{\perp}),
	\]
	where $A^{*},B\in B_1(O_\delta)$.
	
	 Since $A^{*},B\in B_1(O_\delta)$, we have
	\[
	\bigcap_{n=0}^{\infty}\operatorname{ran}A^{n}=\{0\},\quad
	\bigcap_{n=0}^{\infty}\operatorname{ran}(B^{*})^{n}=\{0\}.
	\]
Since $T$ has the bundle-preserving property, there exists $\alpha\in(0,+\infty)$ such that
	\[
	T_{12}e_{B}(\lambda)=e_{A_{L,\alpha}^{-1}}(\lambda), \qquad \lambda\in O_{r},
	\]where $r>0$ is a constant.
Differentiating both sides $n$ times yields:
\[
\begin{cases} 
	T_{12}e_B^{(1)}(\lambda) = e_{A_{L,\alpha}^{-1}}^{(1)}(\lambda) \\ 
	T_{12}e_B^{(2)}(\lambda) = e_{A_{L,\alpha}^{-1}}^{(2)}(\lambda) \\ 
	\vdots \\ 
	T_{12}e_B^{(n)}(\lambda) = e_{A_{L,\alpha}^{-1}}^{(n)}(\lambda)
\end{cases}.
\]
	Furthermore, evaluating the $n$-th derivative of $e_B$ at $\lambda=0$ gives:
	\[
	\begin{cases} 
		e_B^{(1)}(0) = B_R^{-1}f_0 \\ 
		e_B^{(2)}(0) = 2!(B_R^{-1})^2f_0 \\ 
		\vdots \\ 
		e_B^{(n)}(0) = n!(B_R^{-1})^nf_0 
	\end{cases}.
	\]
	Similarly,
\[
\begin{cases} 
	e_{A_{L,\alpha}^{-1}}^{(1)}(0) = \alpha A e_0' \\
	e_{A_{L,\alpha}^{-1}}^{(2)}(0) = 2! \alpha^2 A^2 e_0' \\
	\vdots \\
	e_{A_{L,\alpha}^{-1}}^{(n)}(0) = n! \alpha^n A^n e_0'
\end{cases}.
\]
	Therefore, $T_{12}\big(e_B^{(n)}(0)\big)=e_{A_{L,\alpha}^{-1}}^{(n)}(0)$, i.e.
	\[
	T_{12}(B_R^{-1})^n f_0=\alpha^n A^n e_0'.
	\]
	Then
\begin{align*}
	A_L^{-1}T_{12}(B_R^{-1})(B_R^{-1})^n f_0
	&=A_{L}^{-1}T_{12}(B_{R}^{-1})^{n+1}f_0\\
	&=A_{L}^{-1}\alpha^{n+1}A^{n+1}e_{0}^{\prime}\\
	&=\alpha^{n+1}A^{n}e_{0}^{\prime}.
	\end{align*}
		We may assume $\langle e_0,e_0'\rangle=1$ by Lemma \ref{lem:3.4}. Since $e_0\in\ker A^*$, we obtain
	\[
	\langle A^n e_0',e_0\rangle=\langle e_0^{\prime},(A^*)^n e_0\rangle=0,\quad n\ge 1.\tag{3.1}\label{3.1}
	\]
Combining this with equation (\ref{3.1}), a direct computation shows that
	\[
	T^{-1}\begin{pmatrix}e_0'\\0\end{pmatrix}
	=
	\begin{pmatrix}
		A_L^{-1} & -A_L^{-1}T_{12}(B_R^{-1})\\
		f_0\otimes e_0 & (B_R^{-1})
	\end{pmatrix}
	\begin{pmatrix}e_0'\\0\end{pmatrix}
	=
	\begin{pmatrix}0\\f_0\end{pmatrix},
	\]
	
	\[
	T^{-2}\begin{pmatrix}e_0'\\0\end{pmatrix}
	=
	\begin{pmatrix}
		A_L^{-1} & -A_L^{-1}T_{12}(B_R^{-1})\\
		f_0\otimes e_0 & (B_R^{-1})
	\end{pmatrix}
	\begin{pmatrix}0\\f_0\end{pmatrix}
	=
	\begin{pmatrix}-\alpha e_0'\\(B_R^{-1})f_0\end{pmatrix},
	\]
	
	\[
	T^{-3}\begin{pmatrix}e_0'\\0\end{pmatrix}
	=
	\begin{pmatrix}
		A_L^{-1} & -A_L^{-1}T_{12}(B_R^{-1})\\
		f_0\otimes e_0 & (B_R^{-1})
	\end{pmatrix}
	\begin{pmatrix}-\alpha e_0'\\(B_R^{-1})f_0\end{pmatrix}
	=
	\begin{pmatrix}-\alpha^2 A e_0'\\-\alpha f_0+(B_R^{-1})^2 f_0\end{pmatrix}.
	\]

	Claim that for $k\ge 3$,
	\[
	T^{-k}\begin{pmatrix}e_0'\\0\end{pmatrix}
	=
	\begin{pmatrix}
		-\alpha^{k-1}A^{k-2}e_0'\\
		-\alpha(B_R^{-1})^{k-3}f_0+(B_R^{-1})^{k-1}f_0
	\end{pmatrix}. \tag{3.2}\label{123}
	\]
	Suppose that (\ref{123}) holds for $k=n$.
	Then
	\[
	\begin{aligned}
		T^{-(n+1)}\begin{pmatrix}e_0'\\0\end{pmatrix}
		&=
		\begin{pmatrix}
			A_L^{-1} & -A_L^{-1}T_{12}(B_R^{-1})\\
			f_0\otimes e_0 & (B_R^{-1})
		\end{pmatrix}
		\begin{pmatrix}
			-\alpha^{n-1}A^{n-2}e_0'\\
			-\alpha(B_R^{-1})^{n-3}f_0+(B_R^{-1})^{n-1}f_0
		\end{pmatrix}\\[10pt]
		&=
		\begin{pmatrix}
			-\alpha^{n}A^{n-1}e_0'\\
			-\alpha(B_R^{-1})^{n-2}f_0+(B_R^{-1})^{n}f_0
		\end{pmatrix}.
	\end{aligned}
	\]
	Thus, (\ref{123}) holds for $k=n+1$.
	Hence, (\ref{123}) holds for all $k\ge 3$. By Lemma \ref{lem:3.5}, there exists a sequence $\{g_n\}_{n=0}^{\infty} \subset M$ such that
	\[
	\langle A^n e_0', g_m \rangle = \delta_{n,m} \quad \text{for all } n,m \ge 0.
	\]
	Hence	\[
	\sum_{k=0}^{\infty}\langle T^{-k}e_0^{\prime},g_0\rangle z^k=1-\alpha z^2.
	\]
	By Lemma \ref{lem:3.6}, taking $f=\begin{pmatrix}
		e_{0}^{\prime}\\
		0
	\end{pmatrix},\  g=\begin{pmatrix}
	g_0\\
	0
	\end{pmatrix}$ and $v(z)=\sum_{n=0}^{\infty}\langle T^{-n}f,g\rangle z^{n}=1-\alpha z^2$, we conclude that $T^{-1}$ is intransitive.
\end{proof}

	\begin{corollary}\label{collo3.6}
		Let 
		\[
		T = \begin{pmatrix} A & T_{12} \\ 0 & B \end{pmatrix}\in \mathcal{B}(M\oplus M^{\perp})
		\]
		be an invertible conjugate bi-geometric operator. If there exists an operator
	$Y \in \operatorname{ran} \tau_{A,B}$
	such that
		\[
	\tilde{T}=\begin{pmatrix} A & Y + T_{12} \\ 0 & B \end{pmatrix}
		\]
		has the bundle-preserving property, then $T^{-1}$ is intransitive.
	\end{corollary}
	\begin{proof}
	
			Without loss of generality, 
		let $Y = AX - XB$ and set $G = \begin{pmatrix} I & -X \\ 0 & I \end{pmatrix}$. Clearly, $G^{-1}=\begin{pmatrix} I & X \\ 0 & I \end{pmatrix}$. Then
			\[
			GTG^{-1} = \begin{pmatrix} I & -X \\ 0 & I \end{pmatrix} \begin{pmatrix} A & T_{12} \\ 0 & B \end{pmatrix} \begin{pmatrix} I & X \\ 0 & I \end{pmatrix} = \begin{pmatrix} A & Y + T_{12} \\ 0 & B \end{pmatrix}.
			\]
			By Theorem \ref{thm:3.7}, $GTG^{-1}$ provides an affirmative answer to Halmos's third problem. Hence $T^{-1}$ is intransitive.
		\end{proof}
	
	The following example gives a large number of examples of operators satisfying Corollary \ref{collo3.6}.
		\begin{example}
	
\[
		T = \begin{pmatrix} S & T_{12} \\ 0 & S^* \end{pmatrix}\in B(H^{2}\oplus H^{2}),
		\]
		where $T_{12}$ is defined on $H^2$ by
		\[
		T_{12}x = \langle x, e_0 \rangle e_0 + \langle x, e_1 \rangle e_1 + a_1 \langle x, e_1 \rangle e_2 + \sum_{n=2}^{\infty} \langle x, e_n \rangle \bigl( e_n + a_n e_{n+1} - a_{n-1} e_{n-1} \bigr)
		\]
		for $x \in H^2$, where \(a_0 = 0\), \(\sum_{n=1}^{\infty} |a_n|^2 < \infty\), and \(a_N \neq 0\) for some \(N \in \mathbb{Z}^+\). It follows that $T_{12}\in B(H^{2})$.
		Obviously, 
		\[
		T^{-1} = \begin{pmatrix} S^* & -S^* T_{12} S \\ e_0 \otimes e_0 & S \end{pmatrix}.
		\]
		Furthermore, 
		\[
		T_{12} e_N = e_N + a_N e_{N+1} - a_{N-1} e_{N-1} \neq \alpha e_N
		\]
		for any scalar $\alpha$. Consequently, $T$ does not have the bundle-preserving property.

		Let 
		\[
		X = \sum_{n=1}^{\infty} a_{n} e_n \otimes e_n,
		\]
		and set $Y = SX - XS^*\in \operatorname{ran}\tau_{S,S^{*}}$. Then \[\begin{cases}
			(T_{12}-Y)e_{0}=e_{0}\\
			(T_{12}-Y)e_{1}=e_{1}+a_{1}e_{2}-a_{1}e_{2}=e_{1}\\
			(T_{12}-Y)e_{n}=e_{n}+a_{n}e_{n+1}-a_{n-1}e_{n-1}-(a_{n}e_{n+1}-a_{n-1}e_{n-1})=e_{n},\  n\ge 2.
		\end{cases}\]
		Hence, $T_{12}+(-Y) = I$. Consequently, the operator
		\[\tilde{T}=
		\begin{pmatrix} S & T_{12} +(-Y) \\ 0 & S^* \end{pmatrix}
		\]
		has the bundle-preserving property. By Corollary \ref{collo3.6}, $T^{-1}$ is intransitive.
			\end{example}

		\section{The Third Problem of Halmos for subexponential power growth of Operators}
		This section focuses on Halmos's third problem. For conjugate bi-geometric operators 
		\[
		T=\begin{pmatrix}A&T_{12}\\0&B\end{pmatrix},
		\]
		assuming that the diagonal blocks of 
		$T$ satisfy certain subexponential norm estimates, we prove that \(T^{-1}\) is intransitive (see Theorem \ref{dingli00}). An example is then given to illustrate the limitations of the method (see Example \ref{juxianxing}).

	Based on the concepts and notation from Section 3, namely
	\[
	T=\begin{pmatrix}
		A&T_{12}\\
		0&B
	\end{pmatrix}
	\quad\text{and}\quad
	T^{-1}=\begin{pmatrix}
		A_{L}^{-1}&-A_{L}^{-1}T_{12}B_{R}^{-1}\\
		f_{0}\otimes e_{0}&B_{R}^{-1}
	\end{pmatrix},
	\]
	the main theorem of this section can be stated as follows:

		\begin{theorem}\label{dingli00}
		Let
			\[
			T=\begin{pmatrix}
				A & T_{12}\\
				0 & B
			\end{pmatrix}\in \mathcal{B}(\mathcal{H}_1\oplus\mathcal{H}_2),
			\]
		be an invertible conjugate bi-geometric operator with \(\sigma_{e}(T)=\mathbb{T}\). 
			If there exist constants \(M,b>0\) and \(\delta\in (0,1)\) such that
			\[
			\max\big\{\|A^n\|,\ \|B^n\|,\ \|(A^{*}A)^{-n}\|,\ \|(BB^{*})^{-n}\|\big\}\le M e^{b n^\delta},\quad n\in\mathbb{N}_0,
			\]
			then \(T\), \(T^{-1}\) have nontrivial hyperinvariant subspaces and $A^{*},B\in B_1(\mathbb{D})$.

		\end{theorem}
		
		Before proving the main theorem of this section, we first introduce the following lemmas.

		\begin{lemma}\label{yinli1}
			Let
		\[
		T=\begin{pmatrix}
			A & T_{12}\\
			0 & B
		\end{pmatrix}\in \mathcal{B}(\mathcal{H}_1\oplus \mathcal{H}_{2}).
		\]
		If for all \(n\in\mathbb N_0\),
		\[
		\max\{\|A^n\|,\|B^n\|\}=O(e^{a n^\alpha}),
		\]
		where \(\alpha\in(0,1)\) and \(a>0\), then
		\[
		\|T^n\|=O(e^{b n^\alpha}),
		\]
		where \(b\) is a constant.
		\end{lemma}
		
		\begin{proof}
			For \(n\ge 1\), it is clear that
		\[
		T^n=\begin{pmatrix}
			A^n & \displaystyle\sum_{j=0}^{n-1}A^j T_{12}B^{n-j-1}\\
			0 & B^n
		\end{pmatrix}.
		\]
		Then there exist constants \(C_1,C_2>0\) such that
		\[
		\begin{aligned}
			\|T^n\|
			&\le C_1 e^{an^\alpha}+C_2 e^{an^\alpha}
			+\sum_{j=0}^{n-1}\|A^j\|\|T_{12}\|\|B^{n-j-1}\|\\
			&\le C_1 e^{an^\alpha}+C_2 e^{an^\alpha}
			+\sum_{j=0}^{n-1}C_1 e^{aj^\alpha}\|T_{12}\|C_2 e^{a(n-j-1)^\alpha}\\
			&\le (C_1+C_2)e^{an^\alpha}
			+nC_1C_2\|T_{12}\|e^{2a(n-1)^\alpha}\\
			&\le (C_1+C_2)e^{an^{\alpha}}+C_1C_2\|T_{12}\|e^{\log n+2an^{\alpha}}	.					
		\end{aligned}
		\]
		Clearly, there exists \(N>1\) such that \(\log n\le n^{\alpha}\) for all \(n\ge N\). Hence,
		\[
		\|T^n\|\le M_{1} e^{(2a+1)n^{\alpha}},\ n\ge N,
		\]
		where \(M_{1}=C_1+C_2+C_1C_2\|T_{12}\|\). Consequently, there exist constants \(M\ge 1,b>0\) such that
		\[
		\|T^n\|\le M e^{b n^{\alpha}},\qquad n\in \mathbb{N}_{0}.
		\]
		\end{proof}
		
		\begin{lemma}
			\label{yinli2}
			Let
			\[
			W=\begin{pmatrix}A & T_{12}\\0 & B\end{pmatrix}\in \mathcal{B}(\mathcal{H}_1\oplus \mathcal{H}_{2})
			\]
			be a bounded operator such that
			\[
			\|W^n\| =O(e^{an^\alpha}),n\in \mathbb N_0
			,
			\]where $ 0<\alpha<1,\ a>0$.
			If there exists a rank-one operator \(N\) such that
			\(T=W+N\) is invertible and
			\(\sigma(T)\subset\mathbb{T}\), then there exist \(\beta\in(0,1),\ b>0\) such that
			\[
			\|T^n\|=O(e^{bn^{\beta}}), n\in \mathbb{N}_{0}.
			\]
		\end{lemma}
		
		\begin{proof}
			Obviously,
			\[
			(I-zT)^{-1}=\sum_{n=0}^{\infty} z^n T^n.
			\]
			By the Cauchy integral formula, for \(r\in(0,1)\), we have
			\[
			T^n=\frac{1}{2\pi i}\oint_{|z|=r}\frac{(I-zT)^{-1}}{z^{n+1}}dz, \qquad n\in \mathbb{N}_{0}.
			\]
			Hence
			\[
			\|T^n\|\le r^{-n}\sup_{|z|=r}\|(I-zT)^{-1}\|, \qquad n\in \mathbb{N}_{0}.
			\]
			Hence, we are reduced to estimating \(\|(I-zT)^{-1}\|\).
			From \(\|W^n\|\le Ce^{an^\alpha}, \ C>0\), we get \(r(W)\le1\). Thus for all \(0<|z|<1\), we have \(\left|\frac1z\right|>1\) and \(\frac1z\in\rho(W)\). Hence
			\[
			I-zW=-z\left(W-\frac1z I\right).
			\]
			Since \(W-\frac1z I\) is invertible, \((I-zW)\) is invertible.
			
			\textbf{Step 1}: Compute \((I-zT)^{-1}\).
			Let
			\[
			R(z)=(I-zW)^{-1}=\sum_{k=0}^\infty z^k W^k.
			\]
			Let \(N=u\otimes v\) and
			\[
			d(z)=1-z\langle R(z)u,v\rangle,\qquad z\in\mathbb D.
			\]
			Clearly,
			\[
			|d(z)|\le1+|z|\|N\|\|R(z)\|.
			\]
			Since \(\sigma(T)\subset \mathbb T\), we claim that \(d(z)\) has no zeros on \(\mathbb D\).
			
			If there exists \(z_0\in\mathbb D\) such that \(d(z_0)=0\), let \(x=R(z_0)u\). Then
			\[
			(I-z_0W)x=u\ne0,
			\]
			so
			\[
			\begin{aligned}
				(I-z_0T)x&=(I-z_0W-z_0N)R(z_0)u\\
				&=u-z_0NR(z_0)u\\
				&=u-z_0\langle R(z_0)u,v\rangle u\\
				&=\big(1-z_0\langle R(z_0)u,v\rangle\big)u\\
				&=d(z_0)u\\
				&=0.
			\end{aligned}
			\]
			Thus \(I-z_0T\) is not invertible, which contradicts \(\sigma(T)\subset\mathbb T\). Hence \(d\) has no zeros in \(\mathbb D\).
			
		Obviously, we have
			\[
			\begin{aligned}
				(I-zT)R(z)u&=(I-zW-zN)R(z)u\\
				&=(I-zW)R(z)u-z\langle R(z)u,v\rangle u\\
				&=u-z\langle R(z)u,v\rangle u\\
				&=d(z)u.
			\end{aligned}
			\]
			Let
			\[
			X=R(z)+\frac{z}{d(z)}R(z)NR(z).
			\]
		We have
			\[
			\begin{aligned}
				NR(z)Nx&=NR(z)\langle x,v\rangle u\\
				&=\langle x,v\rangle NR(z)u\\
				&=\langle x,v\rangle\langle R(z)u,v\rangle u\\
				&=a(z)Nx,
			\end{aligned}
			\]where $a(z)=\langle R(z)u,v\rangle$,
			and
			\[
			1-za(z)=d(z),\qquad	NR(z)N=a(z)N.
			\]
			Thus
			\[
			\begin{aligned}
				(I-zT)X
				&=(I-zW-zN)\left(R(z)+\frac{z}{d(z)}R(z)NR(z)\right)\\
				&=I+\frac{z}{d(z)}NR(z)-zNR(z)-\frac{z^2}{d(z)}NR(z)NR(z).
			\end{aligned}
			\]
			Since \(NR(z)NR(z)=a(z)NR(z)\), we get
			\[
			\begin{aligned}
				(I-zT)X
				&=I+\left(\frac{z}{d(z)}-z-\frac{z^2a(z)}{d(z)}\right)NR(z)\\
				&=I+\left(\frac{z(1-za(z))}{d(z)}-z\right)NR(z)\\
				&=I+\left(\frac{zd(z)}{d(z)}-z\right)NR(z)\\
				&=I.
			\end{aligned}
			\]
			Similarly, \(X(I-zT)=I\). Hence \((I-zT)^{-1}=X\). Therefore
			\[
			\|(I-zT)^{-1}\|\le\|R(z)\|+\frac{\|N\|}{|d(z)|}\|R(z)\|^2.
			\]
			\textbf{Step 2}: Estimate \(\dfrac{1}{|d(z)|}\) and \(\|R(z)\|\).
			
			\textbf{First}, we estimate \(\sup_{|z|\le r}\|R(z)\|\).
			
			Let
			\[
			p=\frac{\alpha}{1-\alpha}>0,
			\]
			and let \(\delta=1-r\). Clearly there exists a constant \(C_{\alpha}>0\) such that
			\(k^{\alpha}-\dfrac{\delta}{2a}k\le C_{\alpha}\delta^{-p}\).
			
		It is clear that 
			\[
			\|R(z)\|\le \sum_{k=0}^\infty \|W^k\|\cdot|z|^k
			\le C\sum_{k=0}^\infty e^{ak^\alpha} r^k,\qquad |z|\le r,
			\]
			and \(r=1-\delta\le e^{-\delta}\), so \(r^k\le e^{-k\delta}\). We have
			\[
			\|R(z)\|\le \sum_{k=0}^\infty Ce^{ak^\alpha-k\delta}
			\le \sum_{k=0}^\infty Ce^{aC_\alpha \delta^{-p}-\frac{\delta k}{2}}
			=\sum_{k=0}^\infty C e^{aC_\alpha\delta^{-p}} e^{-\frac{\delta k}{2}}
			=C e^{aC_\alpha\delta^{-p}} \frac{1}{1-e^{-\frac{\delta}{2}}}
			\le \frac{4C}{\delta}e^{aC_\alpha\delta^{-p}}.
			\]
			Let \(t=\frac1\delta>1\). Then	$\frac1\delta\le e^{\frac1p\delta^{-p}}.$
		
			Therefore
			\[
			\|R(z)\|\le 4C e^{\left(\frac1p+aC_\alpha\right)\delta^{-p}},
			\qquad \delta=1-r,\ |z|\le r.
			\]
			So
			\[
			\sup_{|z|\le r}\|R(z)\| \le M_1 e^{M_2\delta^{-p}},
			\qquad M_1=4C,\ M_2=\frac1p+aC_\alpha.\tag{4.1}\label{1.1.1}
			\]
			
			\textbf{Second}, we estimate \(|\dfrac{1}{d(z)}|\).
			
			From (\ref{1.1.1}), we obtain
			\[
			|d(z)|\le (1+\|N\|M_{1})e^{M_{2}(1-r)^{-p}}\le e^{K_{2}(1-r)^{-p}},\ 0<r<1,
			\]
			where \(K_{2}>0\) is a constant, and \(d(0)=1\). Let \(h(z)=\log|d(z)|\). Then \(h(0)=0\), and
			\[
			\frac{1}{|d(z)|}=e^{-h(z)}.
			\]
			Take \(s=\frac{1+r}{2}\), and let
			\[
			M=\frac{K_2+1}{(1-s)^p}.
			\]
			Thus
			\[
			h(z)\le K_2(1-s)^{-p}\le \frac{K_2+1}{(1-s)^p},\qquad z\in B(0,s).
			\]
			Therefore \(u(z)=M-h(z), z\in B(0,s)\) is positive and harmonic. By Harnack's inequality, we have
			\[
			M-h(z)\le \frac{s+r}{s-r}M.
			\]
			So
			\[
			-h(z)\le \frac{s+r}{s-r}M-M=\frac{2rM}{s-r}.
			\]
			Since
			\[
			s-r=\frac{1-r}{2},\qquad 1-s=\frac{1-r}{2},
			\]
			we obtain
			\[
			-h(z)\le \frac{4}{1-r}\cdot\frac{2^p(K_2+1)}{(1-r)^p}
			=\frac{K_3}{(1-r)^{p+1}},
			\]
			where
			$
			K_3=4\cdot2^p(K_2+1)
			$
			is a fixed constant. Therefore
			\[
			\sup_{|z|\le r}\frac{1}{|d(z)|}\le e^{\frac{K_3}{(1-r)^{p+1}}}.
			\]
			Thus
			\[
			\|(I-zT)^{-1}\|\le M_1 e^{M_2(1-r)^{-p}}
			+\frac{\|N\|}{|d(z)|}M_1^2 e^{2M_2(1-r)^{-p}}.\tag{4.2}\label{1.1.2}
			\]
			\textbf{Step 3}: Estimate \(\sup_{|z|\le r}\|(I-zT)^{-1}\|\) and \(\|T^n\|, n\in \mathbb{N}_{0}\).
			
			From (\ref{1.1.2}), we get
			\[
			\sup_{|z|\le r}\|(I-zT)^{-1}\|
			\le M_1 e^{M_2(1-r)^{-p}}
			+\|N\|e^{\frac{K_3}{(1-r)^{p+1}}}M_1^2 e^{2M_2(1-r)^{-p}}.
			\]
			Since
			\[
			\frac{1}{(1-r)^{p+1}}\ge\frac{1}{(1-r)^p},
			\]
			we have
			\[
			\sup_{|z|\le r}\|(I-zT)^{-1}\|
			\le M_1 e^{M_2(1-r)^{-(p+1)}}
			+\|N\|M_1^2 e^{(K_3+2M_2)(1-r)^{-(p+1)}}
			\le M_3 e^{M_4(1-r)^{-(p+1)}},
			\]
			where
			\[
			M_3=M_1+\|N\|M_1^2,\qquad
			M_4=M_2+K_3+2M_2.
			\]
			Therefore
			\[
			\|T^n\|\le r^{-n}\cdot M_3 e^{M_4(1-r)^{-(p+1)}},
			\qquad \text{for all }0<r<1, \qquad n\in \mathbb{N}_{0}.
			\]
			Let
			\[
			q=p+1=\frac{1}{1-\alpha}.
			\]
			For \(n\ge1\), take
			\[
			\delta_n=\frac12 n^{-\frac1{q+1}},\qquad r_n=1-\delta_n\in(0,1).
			\]
			It is well known that 
			\[
			-\log(1-\delta_n)\le2\delta_n.
			\]
			Therefore
			\[
			r_n^{-n}=e^{-n\log r_n}
			=e^{n(-\log(1-\delta_n))}
			\le e^{2n\delta_n}=e^{n^{\frac{q}{q+1}}},
			\]
			and
			\[
			\frac{1}{(1-r_n)^q}=\delta_n^{-q}
			=2^q n^{\frac{q}{q+1}}.
			\]
			Thus
			\[
			\|T^n\|
			\le e^{n^{\frac{q}{q+1}}} M_3 e^{M_4 2^q n^{\frac{q}{q+1}}}
			=M_3 e^{(1+M_42^q)n^{\frac{q}{q+1}}}
			=M_3 e^{M_5 n^{\frac{q}{q+1}}},\qquad n\in \mathbb{N}_{0},
			\]
			where \( M_5=1+M_42^q.\)
			Moreover,
			\[
			\frac{q}{q+1}
			=\frac{\frac{1}{1-\alpha}}{\frac{1}{1-\alpha}+1}
			=\frac{1}{2-\alpha}\in(0,1).
			\]
			Hence \(\|T^{n}\|=O(e^{M_{5} n^{\beta}}),n\in \mathbb{N}_{0},\ \  \beta\in (0,1)\). This completes the proof.
		\end{proof}
		
		\begin{lemma}\cite{MonsalveLopez2020Bishop}          \label{4.12}
			Let $T\in \mathcal{B}(X)$ be a Banach space operator. If $\sigma_\mathrm{e}(T)\neq \sigma(T)$, then $T$ has non-trivial hyperinvariant subspaces. Furthermore, if $\sigma_{\mathrm{le}}(T)\neq \sigma_{\mathrm{re}}(T)$, then the lattice $\mathrm{Lat}(T)$ is rich.
		\end{lemma}

		Now we give the proof of the main theorem of this section.

			\begin{proof}
			
			First, we prove that \(A^{*},B\in B_{1}(\mathbb{D})\).
			For \(\lambda\in \mathbb{D}\), we claim that \(\lambda\notin \sigma_{\mathrm{ap}}(A)\).
			Suppose \(\lambda\in \sigma_{\mathrm{ap}}(A)\). Then there exists a sequence \(\{x_n\}_{n=1}^{\infty}\subset \mathcal{H}_1\) such that \(\|(A-\lambda)x_n\|\to 0\) and $\|x_n\|=1$. Taking a weakly convergent subsequence, still denoted by \(\{x_n\}_{n=1}^\infty\), we assume that \(x_n \rightharpoonup x\).
			
			Since \((A-\lambda)x_n \rightharpoonup (A-\lambda)x\), we have \((A-\lambda)x=0\). By \(A^*\in B_1(O_t), t>0\), it follows that \(x=0\). Let \(z_n=\begin{pmatrix}x_n\\0\end{pmatrix}\). Then \(z_n\rightharpoonup 0\), and
			\[
			\|(T-\lambda)z_n\|=\|(A-\lambda)x_n\|\to 0.
			\]
			Since \(\sigma_e(T)= \mathbb{T}\), \(T-\lambda\) is a Fredholm operator. Let \(P:\mathcal{H}_1\oplus \mathcal{H}_2\to \ker(T-\lambda)\) be the orthogonal projection. Then \(P\) is a finite-rank operator. Moreover, there exists a constant \(c_{\lambda}>0\) such that
			\begin{align*}
				\|(T-\lambda)z_n\| &= \|(T-\lambda)Pz_n + (T-\lambda)(I-P)z_n\| \\
				&= \|(T-\lambda)(I-P)z_n\| \\
				&\ge c_\lambda \|(I-P)z_n\|.
			\end{align*}
			Since \(Pz_n \xrightarrow{\|\cdot\|} P(0)=0\), we have \(\|(I-P)z_n\|\to 1\). This contradicts \(\|(T-\lambda)z_n\|\to 0\).
			Therefore, \(\lambda\notin \sigma_{\mathrm{ap}}(A)\), which implies that \(\operatorname{ran}(A^*-\overline{\lambda}) = \mathcal{H}_1\). Since \(\sigma_e(T^*)\subset \mathbb{T}\), a similar calculation shows that \(B^{*}-\overline{\lambda}\) is bounded below, and hence \(\operatorname{ran}(B-\lambda)=\mathcal{H}_2\) for \(\lambda\in \mathbb{D}\).
			
			Fix \(\lambda_0\in \mathbb{D}\). Since \(B-\lambda_0\) is surjective, there exists a right inverse \(R_0\) such that \((B-\lambda_0)R_0=I\). Note that
			\[
			B-\lambda = (B-\lambda_0)\big[I-(\lambda-\lambda_0)R_0\big].
			\]
			When \(|\lambda-\lambda_0|\cdot\|R_0\|<1\), we get
			\[
			\ker(B-\lambda)=\big[I-(\lambda-\lambda_0)R_0\big]^{-1}\ker(B-\lambda_0).
			\]
			Thus, \(\dim\ker(B-\lambda)=\dim\ker(B-\lambda_0)\), meaning that \(\dim\ker(B-\lambda)\) is locally constant on \(\mathbb{D}\). Since \(\mathbb{D}\) is connected and \(\dim\ker B=1\), we have \(\dim\ker(B-\lambda)=1\) for all \(\lambda\in \mathbb{D}\).
			Similarly, \(\dim\ker(A^*-\lambda)=1\) for all \(\lambda\in \mathbb{D}\).
			Therefore, \(A^*, B\in  B_1(\mathbb{D})\).

			Next, we prove that \(T\) and \(T^{-1}\) have nontrivial hyperinvariant subspaces.

			If \(\sigma(T)\ne \mathbb{T}\), then \(\sigma(T^{-1})\ne \mathbb{T}\). By Lemma \ref{4.12}, \(T\) and \(T^{-1}\) have nontrivial hyperinvariant subspaces.
			
	Now, we consider the case where \(\sigma(T)=\mathbb{T}\).
			Since \(T\) is invertible, there exist \(e_0\in\ker A^*\) and \(f_0\in\ker B\) such that \(\langle T_{12}f_0,e_0\rangle=1\) and $\|e_0\|=1$.
			Set \(P=e_0\otimes e_0\), $L=(A^{*}A)^{-1}A^{*}$, $R=B^{*}(BB^{*})^{-1}$, \(J=f_0\otimes e_0\).
			Clearly, \(LA=I\), \(AL=I-P\), \(BR=I\), \(BJ=0\) and $\|L^{n}\|,\|R^{n}\|\le M^{2}e^{2bn^{\delta}}$.
			
			A direct computation yields
			\[
			PT_{12}J x=\langle x,e_0\rangle PT_{12} f_0=\langle x,e_0\rangle\langle T_{12} f_0,e_0\rangle e_0=Px,\ x\in \mathcal{H}_1,
			\]
			hence \(PT_{12}J=P\). Let
			\[
			W_0=\begin{pmatrix}
				L & -LT_{12}R\\
				0 & R
			\end{pmatrix},\qquad
			Q=\begin{pmatrix}
				-LT_{12}\\
				I
			\end{pmatrix}J\begin{pmatrix}
				I&-T_{12}R
			\end{pmatrix}.
			\]
			By Lemma \ref{yinli1}, there exist constants \(K_1,K_2>0\) such that
			\[
			\|T^n\|\le K_1 e^{K_2 n^\delta},\qquad 
			\|W_0^n\|\le K_1 e^{K_2 n^\delta},\quad n\in\mathbb{N}_0,\ \delta\in(0,1).
			\]
			Note that 
			\[
			TW_0=\begin{pmatrix}
				AL & -ALT_{12}R+T_{12}R\\
				0 & BR
			\end{pmatrix}
			=\begin{pmatrix}
				I-P & PT_{12}R\\
				0 & I
			\end{pmatrix},
			\]
			and
			\[
			\begin{aligned}
				TQ&=\begin{pmatrix}
					-ALT_{12}+T_{12}\\
					B
				\end{pmatrix}J\begin{pmatrix}
					I&-T_{12}R
				\end{pmatrix}\\
				&=\begin{pmatrix}
					PT_{12}J\\
					BJ
				\end{pmatrix}\begin{pmatrix}
					I&-T_{12}R
				\end{pmatrix}\\
				&	=\begin{pmatrix}
					P\\
					0
				\end{pmatrix}\begin{pmatrix}
					I&-T_{12}R
				\end{pmatrix}\\
				&	=\begin{pmatrix}
					P & -PT_{12}R\\
					0 & 0
				\end{pmatrix}.
			\end{aligned}
			\]
			Thus \(T(W_0+Q)=I\). Consequently, \(T^{-1}=W_0+Q\).
			Obviously, for $(x,y)\in \mathcal{H}_{1}\oplus \mathcal{H}_{2}$, we obtain
			\[
			\begin{aligned}
				Q\begin{pmatrix}x\\y\end{pmatrix}
				&=\begin{pmatrix}-LT_{12}\\I\end{pmatrix}J\begin{pmatrix}
					I&-T_{12}R
				\end{pmatrix}\begin{pmatrix}x\\y\end{pmatrix}\\
				&	=\begin{pmatrix}-LT_{12}\\I\end{pmatrix}J(x-T_{12}Ry)\\
				&=\langle x-T_{12}Ry,e_0\rangle\begin{pmatrix}-LT_{12} f_0\\f_0\end{pmatrix}.
			\end{aligned}
			\]
			Set
			\[
			u=\begin{pmatrix}-LT_{12} f_0\\f_0\end{pmatrix},\quad
			v=\begin{pmatrix}e_0\\-R^*T_{12}^*e_0\end{pmatrix}.
			\]
			Then
			\[
			\begin{aligned}
				(u\otimes v)\begin{pmatrix}x\\y\end{pmatrix}
				&=\big(\langle x,e_0\rangle-\langle y,R^*T_{12}^*e_0\rangle\big)u\\
				&=\langle x-T_{12}Ry,e_0\rangle \begin{pmatrix}-LT_{12} f_0\\f_0\end{pmatrix}.
			\end{aligned}
			\]
			Hence \(Q=u\otimes v\).

			Since \(T^{-1}=W_0+Q\) and \(\sigma(T^{-1})=\mathbb{T}\), it follows from Lemma \ref{yinli2} that there exist constants \(K_3,K_4>0\) such that
			\[
			\|T^{-n}\|\le K_{3} e^{K_{4}n^\rho},\qquad \rho=\frac1{2-\delta}\in(0,1),\ n\in\mathbb{N}_0.
			\]
			Consequently, there exist constants \(C_1,C_2>0\) and \(\gamma\in(0,1)\) such that
			\[
			\|T^n\|\le C_1 e^{C_2 |n|^{\gamma}},\quad  n\in\mathbb{Z}.
			\]
			Choose a positive integer \(N\) such that
			\[
			\frac{N}{N+1}>\gamma.
			\]
			Set
			\[
			\theta=\frac{N}{N+1}>\gamma,\qquad s=1+\frac1N.
			\]
			Then \(\theta<1\), and \(\theta=\frac1s\). Let
			\[
			\eta(t)=
			\begin{cases}
				e^{-(1-t^2)^{-N}}, & |t|<1,\\[2mm]
				0, & |t|\ge1.
			\end{cases}
			\]
			\textbf{Step 1}. Prove that there exists a constant \(B>0\) such that
			\[
			\|\eta^{(j)}\|_\infty\le B^j (j!)^s,\qquad j\ge0.
			\]
			
			Fix \(|t|<1\), and set \(v=1-t^2>0\). Take
			$
			\varepsilon=\frac{1}{24N}.
			$
			For \(|z-t|\le \varepsilon v\), set
			$
			\xi=\frac{t^2-z^2}{v}.
			$
			Then
			$
			1-z^2=v(1+\xi).
			$
			Moreover,
			\[
			|\xi|
			\le \frac{|z-t||z+t|}{v}
			\le \varepsilon v\frac{|z+t|}{v}
			=\varepsilon |z+t|
			\le \varepsilon(2|t|+\varepsilon v)
			\le 3\varepsilon
			=\frac1{8N}.
			\]
			Clearly,
			\[
			|\log(1+\xi)|
			\le \sum_{m=1}^{\infty}\frac{|\xi|^m}{m}
			\le \frac{|\xi|}{1-|\xi|}.
			\]
			Thus
			\[
			|N\log(1+\xi)|
			\le \frac{N|\xi|}{1-|\xi|}
			\le N\dfrac{1}{8N-1}=\dfrac{1}{8-\frac{1}{N}}
			\le \frac17.
			\]
			It is well known that
			\[
			|e^w-1|\le e^{|w|}-1.
			\]
			Take \(w=N\log(1+\xi)\). Then
			\[
			\left|(1+\xi)^{-N}-1\right|
			=\left|e^{-N\log(1+\xi)}-1\right|
			\le e^{|N\log(1+\xi)|}-1
			\le e^{\frac{1}{7}}-1
			\le \frac12.
			\]
			And
			\[
			\operatorname{Re}(1+\xi)^{-N}
			=1+\operatorname{Re}\big((1+\xi)^{-N}-1\big)
			\ge 1-\left|(1+\xi)^{-N}-1\right|
			\ge \frac12.
			\]
			Since \(1-z^2=v(1+\xi)\), we have
			\[
			\operatorname{Re}(1-z^2)^{-N}\ge \frac12 v^{-N}.
			\]
			Hence
			\[
			\left|e^{-(1-z^2)^{-N}}\right|
			=e^{-\operatorname{Re}(1-z^2)^{-N}}
			\le e^{-\frac12 v^{-N}}.
			\]
			Let
			\[
			F(z)=e^{-\frac{1}{(1-z^2)^N}},\qquad z\in\mathbb C\setminus\{1,-1\}.
			\]
			Let \(\rho=\varepsilon v\). Since
			\[
			\rho=\varepsilon(1-t^2)
			=\varepsilon(1-|t|)(1+|t|)
			\le 2\varepsilon(1-|t|)
			<1-|t|,
			\]
			and
			\[
			\min\{|t-1|,|t+1|\}=1-|t|,
			\]
			we have \(z\in\mathbb C\setminus\{1,-1\}\) when \(|z-t|<\varepsilon v\). By Cauchy's integral formula, we get
			\[
			\eta^{(j)}(t)=F^{(j)}(t)
			=\frac{j!}{2\pi i}\int_{|z-t|=\varepsilon v}
			\frac{F(z)}{(z-t)^{j+1}}\,dz.
			\]
			Therefore
			\[
			|\eta^{(j)}(t)|
			\le \frac{j!}{2\pi}
			\frac{e^{-\frac12 v^{-N}}}{(\varepsilon v)^{j+1}}
			\cdot 2\pi(\varepsilon v)
			= j!\,(\varepsilon v)^{-j}e^{\frac{-v^{-N}}{2}}.
			\]
			Let \(x=v^{-N}\). Then
			\[
			|\eta^{(j)}(t)|
			\le j!\,\varepsilon^{-j}x^{\frac{j}{N}}e^{\frac{-x}{2}}.
			\]
			Maximizing \(x^{\frac{j}{N}}e^{\frac{-x}{2}}\) gives
			\[
			\sup_{x>0}x^{\frac{j}{N}}e^{\frac{-x}{2}}
			=\left(\frac{2j}{eN}\right)^{\frac{j}{N}}.
			\]
			Since \(j!\ge (\dfrac{j}{e})^j\), we obtain
			\[
			|\eta^{(j)}(t)|
			\le \big[\varepsilon^{-1}(\dfrac{2}{N})^{\frac{1}{N}}\big]^j (j!)^{1+\frac{1}{N}}.
			\]
			And \(|\eta(t)|\le1\). Take
			\[
			B=\max\{1,\varepsilon^{-1}\left(\frac{2}{N}\right)^{\frac{1}{N}}\},
			\]
			which yields the assertion.
			
			For fixed \(j\), as \(t\to\pm1\), \(v\to0\), and
			\[
			v^{-j}e^{\frac{-v^{-N}}{2}}\xrightarrow{v\to 0}0.
			\]
			Thus \(\eta\in C^\infty(\mathbb R)\).
			
			\textbf{Step 2}. Construct nonzero functions \(\varphi\) and \(\psi\).
			
			Take two distinct points in \(\sigma(T)\), denoted by \(e^{i\frac{\pi}{2}},e^{i\frac{3\pi}{2}}\), and set
			\(d=\frac{\pi}{4}\). Let
			\[
			\varphi(e^{it})=\eta\left(\frac{t-\frac{\pi}{2}}{d}\right),\qquad
			\psi(e^{it})=\eta\left(\frac{t-\frac{3\pi}{2}}{d}\right),\qquad t\in[0,2\pi].
			\]
			Clearly,
			\[
			\operatorname{supp}\varphi
			\subset\left[\dfrac{\pi}{4},\dfrac{3\pi}{4}\right],
			\]
			\[
			\operatorname{supp}\psi
			\subset\left[\dfrac{5\pi}{4},\dfrac{7\pi}{4}\right].
			\]
			Thus
			\[
			\varphi\psi=0,
			\]
			and
			\[
			\varphi(e^{i\frac{\pi}{2}})=\psi(e^{i\frac{3\pi}{2}})=e^{-1}\ne0.
			\]
			Moreover,
			\[
			|\varphi^{(j)}(t)|
			=\left|\frac{\eta^{(j)}\big(\frac{t-\frac{\pi}{2}}{d}\big)}{d^j}\right|
			\le \frac{\|\eta^{(j)}\|_\infty}{d^j}
			\le \left(\frac{B}{d}\right)^j (j!)^s
			=B_0^j (j!)^s,
			\]
			where
			$
			B_0=\frac{B}{d}.
			$
			Obviously, \(\varphi,\psi\) are bounded on \(\mathbb T\), so \(\varphi,\psi\in L^2(\mathbb T)\). Let
			\[
			a_k=\frac1{2\pi}\int_0^{2\pi}\varphi(e^{it})e^{-ikt}\,dt,\qquad
			b_k=\frac1{2\pi}\int_0^{2\pi}\psi(e^{it})e^{-ikt}\,dt.
			\]
			For \(k\ne0\), integrating by parts \(j\) times gives
			\[
			|a_k|
			\le \frac{1}{2\pi |k|^j}\int_0^{2\pi}|\varphi^{(j)}(t)|\,dt
			\le \frac{\|\varphi^{(j)}\|_\infty}{|k|^j}
			\le \frac{B_0^j(j!)^s}{|k|^j}
			\le \left(\frac{B_0 j^s}{|k|}\right)^j.
			\]
			Let
			$
			x=\left(\frac{|k|}{2B_0}\right)^{1/s}.
			$
			When \(x\ge1\), take \(j=[ x]\ge1\). Then
			\[
			\frac{B_0 j^s}{|k|}
			\le \frac{B_0\left(\big(\frac{|k|}{2B_0}\big)^{\frac{1}{s}}\right)^s}{|k|}
			=\frac12.
			\]
			Hence
			\[
			|a_k|\le 2^{-j}\le 2e^{-(\log 2)x}.
			\]
			When \(x<1\), i.e. \(\frac{|k|}{2B_0}<1\), there are only finitely many such \(k\). Therefore we may assume
			\[
			|a_k|\le 2e^{-b x}
			=2e^{-b\left(\frac{|k|}{2B_0}\right)^{\frac{1}{s}}}
			=2e^{-c_1|k|^\theta},
			\]
			where
			\[
			\theta=\frac1s<1,\qquad b=\log2,\qquad
			c_1=\frac{b}{(2B_0)^{\frac{1}{s}}}.
			\]
			Similarly,
			\[
			|b_k|\le 2e^{-c_2|k|^\theta},\qquad k\in\mathbb Z,
			\]where \(c_2>0\) is a constant.
			Thus
			\[
			|a_k|+|b_k|\le 4e^{-c|k|^\theta},\qquad k\in\mathbb Z,
			\]
			where \(c=\min\{c_1,c_2\}\). Consequently,
			\[
			\sum_{k\in\mathbb Z}(|a_k|+|b_k|)<+\infty.
			\]
			Let
			\[
			F(e^{it})=\sum_{k\in\mathbb Z}a_k e^{ikt},\qquad t\in [0,2\pi].
			\]
			Then
			\[
			|F(e^{it})|\le \sum_{k\in\mathbb Z}|a_k|<+\infty,
			\]
			so \(\sum_{k\in \mathbb{Z}}a_k e^{ikt}\) converges absolutely and uniformly, hence
			\[
			F(e^{it})\in L^2(\mathbb T),
			\]
			and
			\[
			\hat F(k)=\frac1{2\pi}\int_0^{2\pi}F(e^{it})e^{-ikt}\,dt=a_{k}=\widehat{\varphi}(k),\qquad k\in \mathbb{Z}.
			\]
			Thus \(F=\varphi\), i.e.
			\[
			\varphi(e^{it})=\sum_{k\in\mathbb Z}\hat\varphi(k)e^{ikt}\in L^2(\mathbb T),
			\]
			and the series converges absolutely and uniformly. Similarly,
			\[
			\psi(e^{it})=\sum_{k\in\mathbb Z}b_k e^{ikt}\in L^2(\mathbb T),
			\]
			and it converges absolutely and uniformly.
			
			\textbf{Step 3}. Construct nonzero operators \(U\) and \(V\).
			Let
			\[
			U=\sum_{k\in\mathbb Z}a_k T^k,\qquad
			V=\sum_{k\in\mathbb Z}b_k T^k.
			\]
			Since
			\[
			\sum_{k\in\mathbb Z}\|a_kT^k\|
			\le \sum_{k\in\mathbb Z}2e^{-c|k|^\theta}\|T^k\|
			\le C_1\sum_{k\in\mathbb Z}
			2e^{-c|k|^\theta+C_2|k|^\gamma},
			\]
			and \(\theta>\gamma\), there exists \(N_0\) such that for \(|k|\ge N_0\),
			\[
			C_2|k|^\gamma\le \frac c2|k|^\theta.
			\]
			Hence,
			\[
			\sum_{|k|\ge N_0}
			e^{-c|k|^\theta+C_2|k|^\gamma}
			\le \sum_{|k|\ge N_0}
			e^{-\frac c2|k|^\theta}<+\infty.
			\]
			Therefore \(\sum_{k\in \mathbb{Z}} a_kT^k\) and \(\sum_{k\in \mathbb{Z}} b_kT^k\) both converge, so \(U,V\) are well-defined.
			
			Since
			\[
			\sum_{j,k\in\mathbb Z}|a_jb_k|\|T^{j+k}\|
			\le
			\left(\sum_{j\in\mathbb Z}|a_j|\|T^j\|\right)
			\left(\sum_{k\in\mathbb Z}|b_k|\|T^k\|\right)
			<+\infty,
			\]
			we have
			\[
			UV=VU
			=\sum_{l\in\mathbb Z}
			\left(\sum_{j+k=l}a_jb_{l-j}\right)T^l.
			\]
			Note that
			\[
			\varphi(e^{it})\psi(e^{it})
			=\sum_{j,k\in\mathbb Z}a_jb_k e^{i(j+k)t},
			\]
			hence
			\[
			\widehat{(\varphi\psi)}(l)
			=\sum_{j\in\mathbb Z}a_jb_{l-j}.
			\]
			Since \(\varphi\psi=0\), we have \(\widehat{(\varphi\psi)}(l)=0\), and therefore
			\[
			UV=VU=0.
			\]
			Let
			\[
			\mathcal A=
			\left\{
			f:\ f(e^{it})=\sum_{n\in\mathbb Z}a_n e^{int},\quad
			\|f\|_{\mathcal A}
			=\sum_{n\in\mathbb Z}|a_n|\|T^n\|<+\infty
			\right\}.
			\]
			If \(f,g\in\mathcal A\), then
			\[
			(fg)(e^{it})=\sum_{n,m\in \mathbb{Z}}a_nb_m e^{i(n+m)t}
			=\sum_{k\in\mathbb Z}
			\left(\sum_{m+n=k}a_nb_m\right)e^{ikt}.
			\]
			Clearly
			\[
			\sum_{k\in\mathbb Z}
			\left|\sum_{m+n=k}a_nb_m\right|\|T^k\|
			\le
			\sum_{k\in\mathbb Z}\sum_{m+n=k}|a_nb_m|\|T^k\|
			\le
			\left(\sum_{m\in\mathbb Z}|b_m|\|T^m\|\right)
			\left(\sum_{n\in\mathbb Z}|a_n|\|T^n\|\right)
			<+\infty.
			\]
			Thus
			\[
			\|fg\|_{\mathcal A}\le \|f\|_{\mathcal A}\|g\|_{\mathcal A},
			\]
			so \(fg\in\mathcal A\). It follows that \(\mathcal A\) is a Banach algebra.
			
			Define
			\[
			\Phi:\mathcal A\to B(H),\qquad
			\Phi(f)=f(T)=\sum_{n\in\mathbb Z}a_nT^n.
			\]
			Then $\|\Phi\|\le 1$ and
			\[
			\Phi(fg)
			=\sum_{n,m}a_nb_m T^{m+n}
			=f(T)g(T)
			=\Phi(f)\Phi(g).
			\]
			For \(f(e^{it})=\sum_{n\in\mathbb Z}a_n e^{int}\in\mathcal A\), set
			\[
			P_N(e^{it})=\sum_{n=-N}^N a_n e^{int}.
			\]
			For each \(N=1,2,\ldots\), choose a trigonometric polynomial
			\(q_N\) such that
			\[
			P_N(e^{it})-P_N(\lambda)
			=(e^{it}-\lambda)q_N(e^{it}).
			\]
			Then
			\[
			(e^{it}-\lambda)q_N(e^{it})
			\longrightarrow f(e^{it})-f(\lambda)
			\]
			in the norm of \(\mathcal A\) as \(N\to\infty\).
			Consequently,
			\[
			(T-\lambda I)q_N(T)\to f(T)-f(\lambda)I
			\quad\text{in the operator norm as }N\to\infty.
			\]
			Since \(\lambda\in\sigma(T)\), \(T-\lambda I\) is not invertible. Hence for each \(N\), \((T-\lambda I)q_N(T)\) is not invertible. Since the set of non-invertible elements is closed, the limit \(f(T)-f(\lambda)I\) is also not invertible, i.e.,
			$
			f(\lambda)\in\sigma(f(T)).
			$
			Finally, since \(\varphi\) and \(\psi\) are not identically zero on \(\sigma(T)\), it follows that \(U\) and \(V\) are nonzero. 
			Hence there exist \(u,v\) such that \(Uu\neq 0\) and \(V^*v\neq 0\). Consequently,
			\[
			\sum_{n=0}^{\infty}\langle T^{-n}Uu,V^*v\rangle z^{n}
			=\sum_{n=0}^{\infty}\langle T^{-n}VUu,v\rangle z^{n}=0
			\]
			is a rational function. Therefore, by Lemma \ref{lem:3.6}, \(T^{-1}\) is intransitive. 
			In view of the definitions of \(U\) and \(V\), for \(B\in\{T^{-1}\}'=\{T\}'\), we have \(BU=UB\). Hence \(\overline{\operatorname{ran}U}\) is a nontrivial invariant subspace of \(B\). Consequently, \(T^{-1}\) possesses a nontrivial hyperinvariant subspace.

		\end{proof}

	\begin{example}
	Let	\[
	T=\begin{pmatrix}
		S_\alpha & T_{12}\\
		0 & S_\beta^*
	\end{pmatrix}\in B(H^{2}(\alpha)\oplus H^{2}(\beta))
	\]
	be invertible. If there exist \(M,b>0\), and \(0<\delta<1\) such that
	\[
	\sup_{k\ge 0}\left\{\frac{\alpha_{k+n}}{\alpha_k},\frac{\alpha_k}{\alpha_{k+n}},\frac{\beta_{k+n}}{\beta_k},\frac{\beta_k}{\beta_{k+n}}\right\}\le Me^{bn^\delta},\qquad n\in\mathbb N_0,
	\]
	then \(T\) and \(T^{-1}\) have nontrivial hyperinvariant subspaces.
\end{example}		

\begin{proof}
	Following a similar argument as in the proof of Theorem \ref{dingli00}, we conclude that \(T\) and \(T^{-1}\) have nontrivial hyperinvariant subspaces.

\end{proof}

	The following example shows that the preceding criteria
	do not apply to all conjugate bi-geometric operators.	
		
		\begin{example}\label{juxianxing}
			
			Let \(\mathcal{H}=l^2(\mathbb{N}_0)\), with canonical orthonormal basis \(\{e_n\}_{n\ge0}\).
			Let \(N_j=16^j\) and \(c_j=\dfrac{1}{\log(N_j+2)}\ (j\ge 1)\). Obviously, \(3N_j<N_{j+1}\). Set
			\[
			w_n=
			\begin{cases}
				e^{c_j}, & N_j\le n<2N_j,\\
				e^{-c_j}, & 2N_j\le n<3N_j,\\
				1, & \text{otherwise}.
			\end{cases}
			\]
			Define the operator \(Fe_n=w_n e_{n+1}\), and let \(p_0=1\), \(p_n=\prod\limits_{k=0}^{n-1} w_k\), \(P_n=\max\limits_{0\le k\le n} p_k\). Define the operators \(De_n=\dfrac{p_n^2}{P_n^2}e_n\) and
			\[
			W=\begin{pmatrix}F & D\\ 0 & F^*\end{pmatrix}.
			\]
			A direct calculation gives
			\[
			p_n=
			\begin{cases}
				e^{c_j(n-N_j)}, & N_j\le n<2N_j,\\
				e^{c_j(3N_j-n)}, & 2N_j\le n\le 3N_j,\\
				1, & \text{otherwise}.
			\end{cases}
			\]
			Thus, \(p_n\ge 1\), \(p_{N_j}=p_{3N_j}=1\), and \(p_{2N_j}=e^{c_j N_j}\).
			Since \(e^{-c_1}\le w_n\le e^{c_1}\) and \(w_n\to 1\), we have \(F\in \mathcal{B}(\mathcal{H})\) and \(F\) is bounded below.
			
			From \(p_n\le P_n\), it follows that \(0<\dfrac{p_n^2}{P_n^2}\le 1\). Note that \(De_{2N_j}=\dfrac{p_{2N_j}^2}{P_{2N_j}^{2}}e_{2N_j}=e_{2N_j}\).
			Define \(Le_n=w_n^{-1}e_{n+1}\). A direct calculation yields \(F^*L=I\) and \(L^{*}F=I\).
			
			By the definition of \(w_n\), for any \(\varepsilon>0\), only finitely many terms are not in \([e^{-\varepsilon},e^{\varepsilon}]\). Thus, there exists a constant \(C_{\varepsilon}>0\) such that \(\|F^m\|\le C_{\varepsilon} e^{\varepsilon m}\), which implies \(\|F^m\|^{\frac{1}{m}}\le C_{\varepsilon}^{\frac{1}{m}}e^{\varepsilon}\). Therefore, \(r(F)\le e^{\varepsilon}\) holds for all \(\varepsilon>0\).

			Consequently, \(r(F)\le 1\). Furthermore, since \(\prod\limits_{l=0}^{m-1}w_{k+l}\xrightarrow{k\to\infty}1\) and \(\|F^m\|=\sup\limits_{k\ge0}\prod\limits_{i=0}^{m-1}w_{k+i}\), we have \(\|F^m\|\ge1\). Hence \(r(F)=1\). Similarly, \(r(L)=1\).
			
			For \(|\lambda|<1\), \(I-\lambda L\) is invertible, and \(F^*-\lambda=F^*(I-\lambda L)\). Since \(F^*\) is surjective, \(F^*-\lambda\) is also surjective.
			If \(F^*x=\lambda x\), let \(x=\sum_{n=0}^{\infty} x_n e_n\), then \(w_n x_{n+1}=\lambda x_n\).
			Hence \(x_n=\dfrac{\lambda^n}{p_n}x_0\). Since \(p_n\ge1\), we have \(k_{\lambda}=\sum \dfrac{\lambda^n}{p_n}e_n\in \mathcal{H}\).
			Therefore, \(\ker(F^*-\lambda)=\operatorname{span}\{k_\lambda\}\).
			
			If \(\langle x, k_{\lambda}\rangle=0\) for all \(\lambda\in \mathbb{D}\), then \(\sum\limits_{n=0}^\infty \dfrac{\overline{x_n}}{p_n}\lambda^n=0\).
			This implies \(x_n=0\), and thus \(F^*\in B_1(\mathbb{D})\).
			Therefore, \(W\) is a conjugate bi-geometric operator.
			Obviously, \[\ker L^*=\operatorname{span}\{e_0\},\qquad L^{n}e_0=\dfrac{1}{p_n}e_n,\qquad F^n e_0=p_n e_n.\]
			If \(W\) is a bundle-preserving operator, then there exists \(\alpha>0\) such that
			\[
			D\left(\sum_{n=0}^\infty \frac{\lambda^n}{p_n}e_n\right)=\sum_{n=0}^\infty \alpha^{n}\lambda^n p_n e_n
			\]
			holds in a neighborhood of \(0\), which implies \(\dfrac{1}{P_n^2}=\alpha^n\).
			Furthermore, for \(0\le n\le15\), we have \[w_0=w_1=\dots=w_{15}=1, w_{16}=e^{c_1}=e^{\frac{1}{\log 18}},
			p_0=p_1=\dots=p_{15}=p_{16}=1, p_{17}=e^{c_1}=e^{\frac{1}{\log 18}}.\]
			Since \(\dfrac{1}{P_{1}^{2}}=\alpha\), we have \(\alpha=1\). However, \(\dfrac{1}{P_{17}^{2}}<1\), which contradicts \(\dfrac{1}{P_{n}^{2}}=\alpha^{n}\).
			Thus, \(W\) is not a bundle-preserving operator.
			
			For \(1\le n\le N_j\), we have \(F^n e_{N_j}=e^{n c_j}e_{N_j+n}\).
			Hence \(\|F^n\|\ge e^{n c_j}=e^{\frac{n}{\log(N_j+2)}}\).
			For \(n\ge16\), there exists \(j\) such that \(16^{j-1}<n\le16^j\).
			In this case, \(n\le N_j\) and \(N_j=16^j=16\cdot16^{j-1}<16n\).
			Therefore, \(0<\log(N_j+2)<\log(16n+2)\), and thus
			\[
			\|F^n\|\ge e^{\frac{n}{\log(16n+2)}},\quad(n\ge16).
			\]
			Therefore, \(W\) does not satisfy the conditions of Theorem 4.1.
			
			A direct calculation yields
			\[
			W^{-1}=
			\begin{pmatrix}
				L^* & -L^*DL\\
				e_0\otimes e_0 & L
			\end{pmatrix}.
			\]

			In summary, neither Theorem \ref{thm:3.7} nor Theorem \ref{dingli00} applies to this example. The arguments developed here do not determine whether \(W^{-1}\) has a nontrivial invariant subspace. This example therefore illustrates a limitation of our approach to Halmos's third problem.
				\end{example}
		
		\begin{lemma}\cite{bhowmik2016a}\label{ingatem}
			Let \(\theta : \mathbb{R}^n \to [0, \infty)\) be a decreasing radial function with \(\lim_{|y| \to \infty} \theta(y) = 0\) and
		\[
		S = \sum_{m \in \mathbb{N}} \frac{\theta(m)}{m}.
		\]
	 Let \(f \in L^1(\mathbb{T}^n)\) be such that
		\[
			|\widehat{f}(m)| \leq C e^{-\theta(m)|m|}, 
	\]
		\textit{for all \(m \in \mathbb{Z}^n\). If \(f\) vanishes on any nonempty open set \(U \subset \mathbb{T}^n\) and \(S = \infty\), then \(f(x) = 0\) for all \(x \in \mathbb{T}^n\).}
		
			\end{lemma}
		
		\begin{proposition}\label{fangfafanli}
			Suppose that
			\[
			\|T^n\|=e^{|n|^{\alpha_{|n|}}},\quad n\in \mathbb{Z},
			\]
			where
			\[
			1-\frac{\log(\log(e+|n|))}{\log|n|}\le \alpha_{|n|}<1,\quad |n|\ge 2.
			\]
			If nonzero functions $\varphi,\psi\in C^{\infty}(\mathbb{T})$ satisfy
			\[
			\sum_{n\in \mathbb{Z}}|\widehat{\varphi}(n)|\|T^{n}\|<+\infty,\quad
			\sum_{n\in \mathbb{Z}}|\widehat{\psi}(n)|\|T^{n}\|<+\infty,
			\]
			then $\varphi\psi\not \equiv 0$.
		\end{proposition}
		
		\begin{proof}
		We argue by contradiction. Suppose that \(\varphi\psi\equiv 0\).
			
			From
			\[
			1>\alpha_{|k|}\ge 1-\frac{\log(\log(e+|k|))}{\log|k|},\ |k|\ge 2
			\]
			it is clear that $\alpha_{|k|}\to1$. Moreover,
			\[
			|k|^{\alpha_{|k|}}\ge \frac{|k|}{|k|^{\frac{\log(\log(e+|k|))}{\log|k|}}}=\dfrac{|k|}{\log(e+|k|)}.
			\]
			Let
			\[
			A=\sum_{k\in\mathbb Z}|a_k|\|T^k\|<+\infty,
			\]
			where \(a_k\) are the Fourier coefficients of \(\varphi\). 
			
			Without loss of generality, we may suppose that
			\[
			|a_0|\le A,\qquad |a_{\pm 1}|\le \frac{A}{e}.
			\]
				Then
			\[
			|a_k|\|T^k\|\le A
			\implies
			|a_k|\le \frac{A}{\|T^k\|}
			=A e^{-|k|^{\alpha_{|k|}}}
			\le A e^{-|k|\frac{1}{\log(e+|k|)}}.
			\]
			Let
			\[
			w(x)=\frac{1}{\log(e+|x|)},\ x\in \mathbb{R}.
			\]
			Then
			\[
			|\widehat{\varphi}(k)|= |a_k|\le A e^{-|k|w(k)},\quad k\in\mathbb Z.
			\]
			For \(w(x)=\frac{1}{\log(e+|x|)}\), obviously \(w\) is a radial function, 
			\[
			\lim_{|x|\to\infty}w(x)=0,
			\]and $w \text{ is radially strictly decreasing}.$
			Since for \(m\ge3\),
			\[
			\log(e+m)\le \log m^2=2\log m,
			\]
			we have
			\[
			\sum_{m\in\mathbb N}\frac{w(m)}{m}
			=\sum_{m=1}^{\infty}\frac{1}{m\log(e+|m|)}
			=\sum_{m=1}^{\infty}\frac{1}{m\log(e+m)}
			\ge \sum_{m=3}^{\infty}\frac{1}{2m\log m}
			=+\infty.
			\]

			Since \(\psi\not\equiv 0\) and \(\psi\) is continuous, there exists a nonempty open set \(U\subset\mathbb T\) on which \(\psi\) does not vanish. From \(\varphi\psi\equiv 0\) we obtain \(\varphi\equiv 0\) on \(U\). Hence we conclude that \(\varphi\equiv0\) by Lemma \ref{ingatem}, a contradiction.
		\end{proof}
		
		\begin{remark}
			By Proposition \ref{fangfafanli}, it is impossible to construct nonzero operators
			\[
			U=\sum_{k\in \mathbb{Z}}\widehat{\varphi}(k)T^{k},\quad
			V=\sum_{k\in \mathbb{Z}}\widehat{\psi}(k)T^{k}
			\]
			satisfying $UV=VU=0$, where \(\varphi\) and \(\psi\) are functions with disjoint supports, and
			\[
			\sum_{k\in \mathbb{Z}} \left( |\widehat{\varphi}(k)| + |\widehat{\psi}(k)| \right) < +\infty,\qquad \sum_{k\in \mathbb{Z}} \left( |\widehat{\varphi}(k)| + |\widehat{\psi}(k)| \right)\|T^{k}\| < +\infty.
			\]
		\end{remark}

 \section{Complex vector bundles and the third problem of Halmos}
 
 In this section, we shall characterize the scope of the class of operators for which Halmos's third problem has an affirmative answer.

Inspired by the work of Douglas \cite{cowen1983operators}, we employ invariants derived from the curvature of the holomorphic vector bundles $E_T$ (where $T \in B_1(\Omega)$) to study the third problem of Halmos.

 \begin{lemma}\label{4.1}
 	Let
 	\[
 	T=\begin{pmatrix}
 		A & T_{12}\\
 		0 & B
 	\end{pmatrix}\in B(\mathcal{H}_{1}\oplus \mathcal{H}_{2}),\quad
 	\tilde{T}=\begin{pmatrix}
 		\tilde{A} & \tilde{T}_{12}\\
 		0 & \tilde{B}
 	\end{pmatrix}\in B(\tilde{\mathcal{H}}_{1}\oplus \tilde{\mathcal{H}_{2}}),
 	\]
 	where $A^{*},B,\tilde{A}^{*},\tilde{B}\in B_{1}(O_{\delta})$. Let $S$ and $\tilde{S}$ be the induced operators of $T$ and $\tilde{T}$, respectively. If $S\sim_{X}\tilde{S}$, then
 	\[
 	X=\begin{pmatrix}
 		X_{1} & X_{12}\\
 		0 & X_{2}
 	\end{pmatrix},
 	\]
 	and $A\sim_{s}\tilde{A}$, $B\sim_{s}\tilde{B}$.
 \end{lemma}
 \begin{proof}
 	Note that
 	\[
 	S=\begin{pmatrix}
 		A^{*} & T_{12}\\
 		0 & B^{*}
 	\end{pmatrix},\quad
 	\tilde{S}=\begin{pmatrix}
 		\tilde{A}^{*} & \tilde{T}_{12}\\
 		0 & \tilde{B}^{*}
 	\end{pmatrix}.
 	\]
 	Since
 	\[
 	\begin{pmatrix}
 		X_{1} & X_{12}\\
 		X_{21} & X_{2}
 	\end{pmatrix}
 	\begin{pmatrix}
 		A^{*} & T_{12}\\
 		0 & B^{*}
 	\end{pmatrix}
 	=
 	\begin{pmatrix}
 		\tilde{A}^{*} & \tilde{T}_{12}\\
 		0 & \tilde{B}^{*}
 	\end{pmatrix}
 	\begin{pmatrix}
 		X_{1} & X_{12}\\
 		X_{21} & X_{2}
 	\end{pmatrix},
 	\]
 	we obtain $X_{21}A^{*}=\tilde{B}^{*}X_{21}$. Hence
 	$X_{21}(A^{*}-w)=(\tilde{B}^{*}-w)X_{21}$, so
 	\[
 	X_{21}\ker(A^{*}-w)\subset \ker(\tilde{B}^{*}-w).
 	\]
 	Note that $\sigma_{p}(\tilde{B}^{*})=\emptyset$.
 	Thus $X_{21}e_{A^{*}}(w)=0$, which implies $X_{21}=0$. Consequently,
 	\[
 	X=\begin{pmatrix}
 		X_{1} & X_{12}\\
 		0 & X_{2}
 	\end{pmatrix}.
 	\]
 		Let $Y=X^{-1}$. Then $SY=Y\tilde{S}$. A similar calculation shows that
 	 $Y_{21}=0$, so
 	\[
 	Y=\begin{pmatrix}
 		Y_{1} & Y_{12}\\
 		0 & Y_{2}
 	\end{pmatrix}.
 	\]
 	Since $XY=YX=I$, we have
 	\[
 	\begin{pmatrix}
 		X_{1} & X_{12}\\
 		0 & X_{2}
 	\end{pmatrix}
 	\begin{pmatrix}
 		Y_{1} & Y_{12}\\
 		0 & Y_{2}
 	\end{pmatrix}
 	=
 	\begin{pmatrix}
 		I & 0\\
 		0 & I
 	\end{pmatrix},
 	\]
 	\[
 	\begin{pmatrix}
 		Y_{1} & Y_{12}\\
 		0 & Y_{2}
 	\end{pmatrix}
 	\begin{pmatrix}
 		X_{1} & X_{12}\\
 		0 & X_{2}
 	\end{pmatrix}
 	=
 	\begin{pmatrix}
 		I & 0\\
 		0 & I
 	\end{pmatrix}.
 	\]
 	Therefore,
 	\[
 	\begin{cases}
 		X_{1}Y_{1}=Y_{1}X_{1}=I\\
 		X_{2}Y_{2}=Y_{2}X_{2}=I
 	\end{cases},
 	\]
 	so $X_{1},Y_{1},X_{2},Y_{2}$ are invertible. Since
 	\[
 	\begin{cases}
 		X_{1}A^{*}=\tilde{A}^{*}X_{1}\\
 		X_{2}B^{*}=\tilde{B}^{*}X_{2}
 	\end{cases},
 	\]
 	we obtain
 	\[
 	\begin{cases}
 		A X_{1}^{*}=X_{1}^{*}\tilde{A}\\
 		B X_{2}^{*}=X_{2}^{*}\tilde{B}
 	\end{cases}.
 	\]
 	Hence $A\sim_{s}\tilde{A}$, $B\sim_{s}\tilde{B}$.
 	\end{proof}

 \begin{theorem}\label{fujihe}
 	Let
 	\[
 	T=\begin{pmatrix}
 		A & T_{12}\\
 		0 & B
 	\end{pmatrix}\in B(\mathcal{H}_{1}\oplus \mathcal{H}_{2}),
 	\]
 	and suppose that $T^{-1}$ is intransitive. Let
 	\[
 	\tilde{T}=\begin{pmatrix}
 		\tilde{A} & \tilde{T}_{12}\\
 		0 & \tilde{B}
 	\end{pmatrix}\in B(\tilde{\mathcal{H}}_{1}\oplus \tilde{\mathcal{H}_{2}}),
 	\]
 	where $A^{*},\tilde{A}^{*},B,\tilde{B}\in B_{1}(O_{\delta})$.
 	Denote by $S_{T}$ and $S_{\tilde{T}}$ the induced operators of $T$ and $\tilde{T}$, respectively.
 	Let $\gamma(w),\tilde{\gamma}(w)$ and $\delta(w),\tilde{\delta}(w)$ be spanning sections for $\ker(S_{T}-w),\ker(S_{\tilde{T}}-w)$ and $\ker(S_{T}^{*}-w),\ker(S_{\tilde{T}}^{*}-w)$, respectively.
 	If the following conditions hold:
 	\begin{enumerate}
 		\item $\mathcal{K}_{A^{*}}(w)=\mathcal{K}_{\tilde{A}^{*}}(w)$.
 		
 		\item $\mathcal{K}_{B}(w)=\mathcal{K}_{\tilde{B}}(w)$.
 		
 		\item
 		$
 		\dfrac{\bigl|\langle\gamma(w_{1}),T_{12}\delta(\overline{w}_{2})\rangle\bigr|}
 		{\|\gamma(w_{1})\|\,\|\delta(\overline{w}_{2})\|}
 		=
 		\dfrac{\bigl|\langle\tilde{\gamma}(w_{1}),\tilde{T}_{12}\tilde{\delta}(\overline{w}_{2})\rangle\bigr|}
 		{\|\tilde{\gamma}(w_{1})\|\,\|\tilde{\delta}(\overline{w}_{2})\|},
 		$
 		for all $w,w_{1},w_{2}\in O_{\delta}$.
 	\end{enumerate}
 	then $\tilde{T}$ is invertible and $\tilde{T}^{-1}$ is intransitive.
 \end{theorem}
 \begin{proof}
 From conditions (1) and (2), we obtain \(A^{*}\sim_u \tilde{A}^{*}\) and \(B\sim_u \tilde{B}\). Consequently, \(B^{*}\sim_u \tilde{B}^{*}\).
 Together with condition (3) and Remark \ref{rem:2.7}, we get $S_{T}\sim_{u}S_{\tilde{T}}$.
 	Since $S_{T},S_{\tilde{T}}\in B_{1,1}(O_{\delta})$, it follows from Lemma \ref{4.1} that there exists a unitary operator
 	\[
 	U=\begin{pmatrix}
 		U_{1} & U_{12}\\
 		0 & U_{2}
 	\end{pmatrix}
 	\]such that $US_{T}=S_{\tilde{T}}U$.
 	
 From the proof of Lemma \ref{4.1}, it follows that $U^{-1}$ is also an upper triangular matrix. Moreover, since $U$ is unitary, we have $U^* = U^{-1}$. Consequently, $U$ is a diagonal operator, i.e.\[
 	U=\begin{pmatrix}
 	U_{1} & 0\\
 	0 & U_{2}
 \end{pmatrix}.
 \]
 The identity $US_{T}=S_{\tilde{T}}U$ yields
 	\[
 	\begin{cases}
 		U_{1}A^{*}=\tilde{A}^{*}U_{1}\\[4pt]
 		U_{2}B^{*}=\tilde{B}^{*}U_{2}\\[4pt]
 		U_{1}T_{12}=\tilde{T}_{12}U_{2}
 	\end{cases}.
 	\]
 	Therefore, we have
 	\[
 	U_{1}A=\tilde{A}U_{1},\quad U_{2}B=\tilde{B}U_{2}.
 	\]
 	Consequently, 
 	\[
 	\begin{pmatrix}
 		U_{1} & 0\\
 		0 & U_{2}
 	\end{pmatrix}
 	\begin{pmatrix}
 		A & T_{12}\\
 		0 & B
 	\end{pmatrix}
 	=
 	\begin{pmatrix}
 		U_{1}A & U_{1}T_{12}\\
 		0 & U_{2}B
 	\end{pmatrix}
 	=
 	\begin{pmatrix}
 		\tilde{A} & \tilde{T}_{12}\\
 		0 & \tilde{B}
 	\end{pmatrix}
 	\begin{pmatrix}
 		U_{1} & 0\\
 		0 & U_{2}
 	\end{pmatrix}.
 	\]
 	Thus $T\sim_{u}\tilde{T}$, which implies that $\tilde{T}^{-1}$ is intransitive.
 	 \end{proof}

 \begin{remark}
 	The similarity of induced operators of conjugate bi-geometric operators does not imply the similarity of the original conjugate bi-geometric operators. A counterexample is as follows.
 	
 Let $e_n(z)=z^n$, $n\ge0$, be the standard orthonormal
 basis of $H^2$, and let $S=M_z$. Set  	\[
 	E_{ij}=e_i\otimes e_j,\  i,j\in \mathbb{N}_{0}, \qquad P_0=E_{00}.
 	\]
 	Take
 	\[
 	T_0=\begin{pmatrix}
 		S & P_0\\
 		0 & S^*
 	\end{pmatrix},\qquad
 	T_1=\begin{pmatrix}
 		S & P_0-2E_{01}\\
 		0 & S^*
 	\end{pmatrix}.
 	\]
 	Clearly, \(T_0\) and \(T_1\) are conjugate bi-geometric operators.
 	
 	The corresponding induced operators are
 	\[
 	S_0=\begin{pmatrix}
 		S^* & P_0\\
 		0 & S
 	\end{pmatrix},\qquad
 	S_1=\begin{pmatrix}
 		S^* & P_0-2E_{01}\\
 		0 & S
 	\end{pmatrix}.
 	\]
 	Let \(F=2(E_{11}+E_{20})\) and
 	\[
 	G=\begin{pmatrix}
 		I & F\\
 		0 & I
 	\end{pmatrix}.
 	\]
 	We have
 	\[
 	\begin{aligned}
 		G S_0 G^{-1}
 		&=\begin{pmatrix}
 			I & F\\
 			0 & I
 		\end{pmatrix}
 		\begin{pmatrix}
 			S^* & P_0\\
 			0 & S
 		\end{pmatrix}
 		\begin{pmatrix}
 			I & -F\\
 			0 & I
 		\end{pmatrix}\\[4pt]
 		&=\begin{pmatrix}
 			S^* & -S^*F+P_0+FS\\
 			0 & S
 		\end{pmatrix}\\[4pt]
 		&=S_1.
 	\end{aligned}
 	\]
 	Thus \(S_0\sim_s S_1\).
 	
 	Take \(\displaystyle h=\sum_{n=0}^{\infty}\frac{e_n}{2^n}\in H^2\). Then
 	$
 	S^*h=\frac12 h,\ (P_0-2E_{01})h=0.
 $
 	Therefore,
 	\[
 	T_1\begin{pmatrix}0\\ h\end{pmatrix}
 	=\frac12\begin{pmatrix}0\\ h\end{pmatrix},
 	\]
 	so \(\frac12\in\sigma_p(T_1)\). On the other hand, \(T_0\) is unitary, so
 	\(\sigma(T_0)\subset\mathbb T\), and therefore it is impossible that \(T_0\sim_s T_1\).

 \end{remark}

 \section{Appendix}
 
For a conjugate bi-geometric operator $T$, an alternative approach to determining whether $T^{-1}$ is intransitive is to decide whether $T^{-1}$ is similar to an upper-triangular operator matrix. Accordingly, we state the following open problem.
 
  \textbf{Problem.} Consider the operator
  \[
  T^{-1}=
  \begin{pmatrix}
  	A_L^{-1} & -A_L^{-1}T_{12}B_R^{-1}\\
  	f_0\otimes e_0 & B_R^{-1}
  \end{pmatrix},
  \]
where $A^*, B \in B_1(O_\delta)$, and the nonzero vectors
$e_0 \in \ker A^*$ and $f_0 \in \ker B$ are chosen so that
$\langle T_{12}f_0,e_0\rangle=1$.
Here $A_L^{-1}$ and $B_R^{-1}$ denote the left and right
inverses of $A$ and $B$ associated with $T$, respectively,
as specified in Lemma \ref{keni}.
  
   When is $T^{-1}$ similar to some upper‑triangular operator matrix
  $\begin{pmatrix}A_1 & B_1\\ 0 & C_1\end{pmatrix}$?
  
  Consider the conjugate bi-geometric operator
  \[
  T=
  \begin{pmatrix}
  	A & T_{12}\\
  	0 & B
  \end{pmatrix}\!\!\!\!\begin{matrix}
  M\\
  \ \ M^{\perp}
  \end{matrix}
  \]where $A^{*},B\in B_1(O_\delta)$. We obtain
  \[
  T^{-1}=
  \begin{pmatrix}
  	A_L^{-1} & -A_L^{-1}T_{12}B_R^{-1}\\
  	f_0\otimes e_0 & B_R^{-1}
  \end{pmatrix}.
  \]
  Suppose there exists an invertible upper‑triangular operator
  \[
  X=
  \begin{pmatrix}
  	X_1 & X_{12}\\
  	0 & X_2
  \end{pmatrix}
  \]
  such that
  \[
  XT^{-1}=
  \begin{pmatrix}
  	A_1 & B_1\\
  	0 & C_1
  \end{pmatrix}X.
  \]
Consequently, we have
\[
X_2 B_R^{-1} = C_1 X_2,\qquad X_2 f_0 = 0.
\]
Thus $X_2 B_R^{-1} f_0 = C_1 X_2 f_0 = 0$. By induction, it follows that
\[
X_2 (B_R^{-1})^n f_0 = 0 \quad \text{for all } n \in \mathbb N.
\]
Since $B\in B_1(O_\delta)$, we have \[\bigvee_{n=0}^{\infty}\{(B_R^{-1})^nf_0\}=M^\perp.\] Consequently, $X_2=0$, which contradicts the invertibility of $X$.

The preceding argument shows that $T^{-1}$ cannot be transformed
into an upper triangular operator matrix by an invertible operator
that is itself upper triangular with respect to the original
decomposition. This does not rule out triangularization by a
general similarity transformation.
And Theorems \ref{thm:3.7} and \ref{dingli00} provide sufficient
 conditions for the intransitivity of $T^{-1}$.

\printbibliography 
\end{document}